\documentclass[12pt]{amsart}

\usepackage{color}
\usepackage[utf8]{inputenc}
\usepackage{amsfonts,amsmath}
\usepackage{amssymb,latexsym}

\usepackage{etoolbox}
\patchcmd{\section}{\scshape}{\bfseries\Large}{}{}
\makeatletter
\renewcommand{\@secnumfont}{\bfseries\Large}
\makeatother

\newtheorem{theorem}{Theorem}
\newtheorem{lemma}[theorem]{Lemma}
\newtheorem{proposition}[theorem]{Proposition}
\theoremstyle{remark}

\newcommand{\real}{\mathbb{R}}
\newcommand{\R}{\mathbb{R}}

\renewcommand{\leq}{\leqslant}
\renewcommand{\geq}{\geqslant}

\def\tiu{\widetilde{u}}

\def\Z{\partial_Z }
\def\zm{\partial^m_Z }
\def\N{\mathbb{N}}
\newcommand\ZZ[1]{\partial_{Z_{{#1}}}}
\def\cz{\mathscr{Z}}
\def\cz{\mathcal{Z}}

\def\al{\alpha}
\def\be{\beta}
\def\phi{\varphi}

\def\PP{\mathbb{P}}

\DeclareMathOperator{\dive}{div}
\DeclareMathOperator{\curl}{curl}

\def\ep{\varepsilon}
\def\om{\omega}
\def\Om{\Omega}
\def\bdry{{\partial\Omega}}
\def\oma{\omega^\alpha}
\def\dt{\partial_{t}}
\def\tan{_{tan}}
\def\nor{_{nor}\,}
\newcommand\nl[2]{\|#2\|_{L^{#1}}}
\newcommand\nh[2]{\|#2\|_{H^{#1}}}

\newcommand\hco[2]{\|#2\|_{H^{#1}_{co}}}
\newcommand\bighco[2]{\Bigl\|#2\Bigr\|_{H^{#1}_{co}}}
\newcommand\hcom[2]{\|#2\|_{H^{#1}_{co}}}
\newcommand\xm[2]{\|#2\|_{X^{#1}}}
\newcommand\ym[2]{\|#2\|_{Y^{#1}}}
\newcommand\lipco[1]{\|#1\|_{W^{1,\infty}_{co}}}
\newcommand\uco[2]{\|#2\|_{W^{#1,\infty}_{co}}}
\newcommand\biguco[2]{\Bigl\|#2\Bigr\|_{W^{#1,\infty}_{co}}}

\begin{document}
\title{Uniform time of existence for the alpha Euler equations}
\author[Busuioc, Iftimie, Lopes Filho and Nussenzveig Lopes]{{A. V. Busuioc,} {D. Iftimie,}
{M. C. Lopes Filho and}
{H. J. Nussenzveig Lopes}
}

\address[A.V. Busuioc]{Université de Lyon, Université de Saint-Etienne  --
CNRS UMR 5208 Institut Camille Jordan --
Faculté des Sciences --
23 rue Docteur Paul Michelon --
42023 Saint-Etienne Cedex 2, France}
\email{valentina.busuioc@univ-st-etienne.fr}

\address[D. Iftimie]{Universit\'e de Lyon, Universit\'e Lyon 1 --
CNRS UMR 5208 Institut Camille Jordan --
43 bd. du 11 Novembre 1918 --
Villeurbanne Cedex F-69622, France.}
\email{iftimie@math.univ-lyon1.fr}
\urladdr{http://math.univ-lyon1.fr/\~{}iftimie}

\address[M. C. Lopes Filho]{Instituto de Matem\'atica\\Universidade Federal do Rio de Janeiro\\Cidade Universit\'aria -- Ilha do Fund\~ao\\Caixa Postal 68530\\21941-909 Rio de Janeiro, RJ -- BRAZIL.}
\email{mlopes@im.ufrj.br}

\address[H. J. Nussenzveig Lopes]{Instituto de Matem\'atica\\Universidade Federal do Rio de Janeiro\\Cidade Universit\'aria -- Ilha do Fund\~ao\\Caixa Postal 68530\\21941-909 Rio de Janeiro, RJ -- BRAZIL.}
\email{hlopes@im.ufrj.br}

\begin{abstract}
We consider the $\alpha$-Euler equations on a bounded three-dimensional domain with frictionless Navier boundary conditions. Our main result is the existence of a strong solution on a positive time interval, uniform in $\alpha$, for $\alpha$ sufficiently small. Combined with the convergence result in \cite{busuioc_incompressible_2012}, this implies convergence of solutions of the $\al$-Euler equations to solutions of the incompressible Euler equations when $\al \to 0$. In addition, we obtain a new result on local existence of strong solutions for the incompressible Euler equations on bounded three-dimensional domains. The proofs are based on new {\it a priori} estimates in conormal spaces.
\end{abstract}

\maketitle

\date{\today}


\section{Introduction}
The $\al$-Euler equations, $\al > 0$, are a system of equations given by:
\begin{equation}\label{maineq}
  \partial_t (u-\al\Delta u)+u\cdot\nabla(u-\al\Delta u)+\sum_j(u-\al\Delta u)_j\nabla u_j=-\nabla p, \qquad \dive u=0,
\end{equation}
where $u=(u_1,u_2,u_3)$ is the velocity and $p$ is the scalar pressure.

These equations arise as the zero-viscosity case of the second grade fluids, a model of non-Newtonian fluids introduced in \cite{dunn_thermodynamics_1974} as one among a hierarchy of models of viscoelastic fluids called fluids of differential type.
The $\al$-Euler equations are also used as a sub-grid scale model in turbulence and have been found to possess deep geometric
significance, see \cite{marsden_geometry_2000}.

Note that, if we formally set $\al=0$ in \eqref{maineq}, then we obtain the incompressible Euler equations:
\begin{equation}\label{eulereq}
  \partial_t u+u\cdot\nabla u=-\nabla p, \qquad \dive u=0,
\end{equation}
since $\sum_j u_j \nabla u_j$ is a gradient and can be absorbed by the pressure.

Existence of smooth solutions for system \eqref{maineq} has been established locally in time, in several contexts, see \cite{busuioc_second_2002}, \cite{marsden_geometry_2000,shkoller_boundary_2000} and \cite{busuioc_second_2003}. Global existence, however, is an open problem, a situation which parallels the outstanding open problem of existence of smooth solutions for the $3$-dimensional Euler equations \eqref{eulereq}.
The main concern of the present work is the existence of smooth solutions of the $\al$-Euler equations \eqref{maineq} up to a time which is uniform with respect to $\al$.

This problem needs to be considered in several fluid domains. In the case of flow in all of $\real^3$, existence of a smooth solution, with smooth initial data, was established for a time at least as long as the time of existence for $3D$ Euler, see \cite{linshiz_convergence_2010}. For flow in a smooth, bounded domain with no-slip boundary conditions ($u=0$), the problem remains open. In this paper we consider flow in a smooth, bounded domain $\Om \subset \real^3$, with frictionless Navier boundary conditions, i.e.
\begin{equation}\label{navier}
u\cdot n=0,\quad [D(u)n]\bigl|_{tan}=0  \quad\text{on }\partial\Om,
\end{equation}
where $D(u)$ is the deformation tensor defined by $D(u)=\frac12\bigl((\nabla u)+(\nabla u)^t\bigr)$ and the subscript ``tan"  denotes the tangential part.
Our main result is to show that, given a sufficiently smooth initial velocity $u_0$, there exists a solution of \eqref{maineq} satisfying \eqref{navier}, for a time which is independent of $\al$.

This analogous question may be posed for the Navier-Stokes equations,
\begin{equation}\label{nunseq}
  \partial_t u+u\cdot\nabla u=-\nabla p + \nu \Delta u, \qquad \dive u=0,
\end{equation}
namely, existence of  solutions for a time independent of viscosity $\nu$. For this problem, in the case of the flow in full-space it is classical that, if the initial velocity is sufficiently smooth, then a smooth solution exists up to a time uniform with respect to $\nu$, see \cite{majda_vorticity_2002}. For flows in a smooth, bounded domain, under no slip boundary conditions, a uniform-in-$\nu$ time of existence is an open problem. In recent work, N. Masmoudi and F. Rousset considered the case of flows in a smooth, bounded fluid domain under Navier boundary conditions with friction coefficient $\beta \in \real$:
\begin{equation}\label{navierbeta}
u\cdot n=0,\quad [D(u)n]\bigl|_{tan} + \beta u\bigl|_{tan}=0 \quad\text{on }\partial\Om.
\end{equation}
They showed, in \cite{masmoudi_uniform_2012-1}, that there is a time-of-existence which is uniform with respect to $\nu$ and, in addition, that the vanishing viscosity limit holds. Their analysis relied on estimates in conormal Sobolev spaces, where regularity is measured only via tangential derivatives. The conormal spaces are a well-known tool in the study of symmetric hyperbolic systems, see for instance \cite{gues_probleme_1990,nishitani_regularity_2000}.
In the present article we adapt the ideas developed in \cite{masmoudi_uniform_2012-1} to our problem.

We draw two important corollaries from the analysis contained in this paper. The first is associated with the limit as $\al \to 0$ of solutions of the $\al$-Euler equations.
To contextualize this first corollary we briefly survey the known results regarding
the limiting behavior of $\al$-Euler as $\al \to 0$. In the absence of boundaries the convergence to the Euler equations is relatively simple and was proved in \cite{linshiz_convergence_2010}; see also \cite{busuioc_incompressible_2012}. In the presence of a boundary, and under the no-slip boundary condition, the convergence was treated in \cite{lopes_filho_convergence_2015}, but only in the $2D$ case; the three-dimensional case remains open.

In the case of the frictionless Navier boundary conditions, the authors established, see Theorem 5 in \cite{busuioc_incompressible_2012}, the $L^2$-convergence, as $\al \to 0$, under  the additional hypothesis that weak $H^1$ solutions for the $\al$-Euler equations exist on a time interval independent of $\al$. This hypothesis is known to hold true in dimension two and also for axisymmetric solutions in dimension three. Now, putting together the main result in the present work with \cite[Theorem 5]{busuioc_incompressible_2012}, yields a complete proof of $L^2$-convergence, as $\al \to 0$, for a general bounded three-dimensional smooth domain.

The second corollary is a new local-in-time existence result for the $3D$-Euler equations in a conormal Sobolev space. We remark that this result is an improvement with respect to the existence part of \cite[Theorem 2]{masmoudi_uniform_2012-1}.

In addition to the uniform-in-$\al$ time-of-existence and the two corollaries mentioned above, the proof of our main result requires certain elliptic regularity estimates in conormal spaces, something which is not available in the literature in our context, and which we establish here. We also present a new approximation procedure, within the class of divergence free vector fields with sufficient regularity, measured in conormal spaces.

Next, we give precise statements of our results. We denote by $H^m_{co}$ the space of square integrable functions such that all tangential derivatives of order $\leq m$ are also square integrable. The space $X^m$ is the same as $H^m_{co}$ except that we allow one of the derivatives to be non-tangential. The $W^{m,\infty}_{co}$ is the space of bounded functions  such that all tangential derivatives of order $\leq m$ are also bounded.(Precise definitions of $H^m_{co}$, $X^m$  and $W^{m,\infty}_{co}$ will be given in Section \ref{conormal}.) Let us also introduce $\oma=\curl u-\al\Delta\curl u$.

Our main result is the following theorem. We will assume in the sequel that $\Om$ is a smooth and bounded  open set of $\R^3$.
\begin{theorem}[uniform time of existence]\label{uniformtime}
Let $u_0$ be divergence free and  verifying the Navier boundary conditions \eqref{navier}. Assume moreover that $u_0\in L^2$ and $\oma_0\in H^{m-1}_{co}\cap W^{1,\infty}_{co}$ where $m\geq5$. There exists $\al_0>0$ and a time $T>0$ independent of $\al$ such that for all $0<\al<\al_0$ there exists a solution $u$ of \eqref{maineq} and \eqref{navier} bounded in $L^\infty(0,T;X^m\cap W^{1,\infty})$ independently of $\al$.  Moreover, the time existence $T$ depends only on $\nl2{u_0}$, $\lipco{\oma_0}$ and $\hco {m-1}{\oma_0}$.
\end{theorem}

Combining this theorem with \cite[Theorem 5]{busuioc_incompressible_2012} yields, as mentioned, a result on convergence to a solution of the Euler equations in $\Om$ subject to the non-penetration boundary condition
\begin{equation}\label{tangent}
u\cdot n=0,  \quad\text{on }\partial\Om.
\end{equation}

\begin{theorem}[convergence]\label{convergence}
Let $u_0$ be divergence free and  verifying the Navier boundary conditions \eqref{navier}. Assume that $u_0\in H^3$ and $\curl u_0, \Delta \curl u_0\in H^4_{co}\cap W^{1,\infty}_{co}$. Let $\overline u$ be the solution of the incompressible Euler equations \eqref{eulereq} and \eqref{tangent} with initial data $u_0$. There exists some time $T$ independent of $\al$ and a solution $u^\al$ of \eqref{maineq} and \eqref{navier} on $[0,T]$ with initial data $u_0$ such that
\begin{equation*}
\lim_{\al\to0}\|u^\al-\overline u\|_{L^\infty(0,T;L^2)}=0.
\end{equation*}
\end{theorem}

As a particular case of Theorem \ref{uniformtime} (case $\al=0$) we obtain a new existence result for the incompressible Euler equations.

\begin{theorem}\label{euler}
Let $u_0$ be divergence free, tangent to the boundary and such that $u_0\in X^4$ and $\curl u_0\in W^{1,\infty}_{co}$. Then there exists a unique local in time solution $u$ of  the incompressible Euler equations \eqref{eulereq} and \eqref{tangent} with initial data $u_0$ such that $u\in L^\infty(0,T;X^4\cap W^{1,\infty})$ and $\curl u\in L^\infty(0,T;W_{co}^{1,\infty})$.
\end{theorem}

As noted, Theorem \ref{euler} is an improvement over the existence result for the Euler equations obtained in \cite{masmoudi_uniform_2012-1}, as we assume $u_0\in X^4$ while in \cite{masmoudi_uniform_2012-1} the authors need $u_0\in X^7$. Note also that, compared to the classical $H^3$ existence result for the Euler equation, Theorem \ref{euler} requires only one additional derivative.

The structure of the paper is the following. In the next section we introduce notation and prove an identity related to the Navier boundary conditions. In Section 3 we give precise definitions of the conormal Sobolev spaces and we discuss relevant properties. Section 4 contains elliptic regularity estimates in conormal spaces. We prove, in Section 5,  \textit{a priori} estimates, in conormal spaces, on the solutions of \eqref{maineq}. In Section 6 we construct a sequence of approximate solutions and we use the \textit{a priori} estimates from Section 5  to obtain Theorems \ref{uniformtime} and \ref{euler}. We add some concluding remarks in Section 7.

\section{Some notations and preliminary results}
Let
\begin{equation*}
\om=\curl u\qquad\text{and}\qquad \oma=\om-\al\Delta\om.
\end{equation*}
Applying the curl to relation \eqref{maineq} implies the following equation for the vorticity $\oma$:
\begin{equation}\label{eqvort}
  \dt\oma+u\cdot\nabla\oma-\oma\cdot\nabla u=0.
\end{equation}

We denote by $n$ a smooth vector field defined on $\overline \Om$ such that its restriction to the boundary is the unitary exterior normal to the boundary. We assume moreover that $\|n\|=1$ in a small neighborhood of the boundary $\Om_\delta=\{x\in\overline\Om\ ;\ d(x,\partial\Om)\leq\delta\}$. We define $\partial_n=n\cdot\nabla$ inside $\Om$ too.  We introduce a smooth function $d:\overline\Om\to\R_+$ such that $d$ never vanishes in $\Om$ and such that $d(x)=d(x,\partial\Om)$ for all $x\in\Om_\delta$. In other words,  $d$ is a smooth version of $d(x,\partial\Om)$.

For a vector field $w$ we define
\begin{equation*}
w_{tan}=w\times n\quad\text{and}\quad w_{nor}=w\cdot n.
\end{equation*}
We observe that for any vector fields  $w$ and $\widetilde w$ we have the following relation:
\begin{equation*}
w\cdot\widetilde w=w_{tan}\cdot\widetilde w_{tan}+w_{nor}\widetilde w_{nor}\quad\text{on }\Om_\delta.
\end{equation*}
More generally, the above relation holds true everywhere if one multiplies  the LHS by $\|n\|^2$.

We now recall  some identities related to the Navier boundary conditions. The proof is included for completeness' sake.

\begin{lemma}\label{ident}
Suppose that $u$ is divergence free and verifies the Navier boundary conditions \eqref{navier}. Then
\begin{equation}\label{ident0}
\om\times n =-2n\times\sum_iu_i(n\times\nabla)n_i \equiv F(u)\quad\text{on }\bdry
\end{equation}
and
\begin{equation*}
n\cdot \partial_n\om=(n\times\nabla)\cdot F(u)  -  (n\times \nabla)\cdot u\,\dive n\equiv G(u,(n\times \nabla)u)\quad\text{on }\bdry
\end{equation*}

\end{lemma}
\begin{proof}
Relation \eqref{ident0} was proved in  \cite[Eqn. (14)] {busuioc_second_2003}. Next, we use that $\om$ is divergence free and write
\begin{align*}
(\partial_n\om)\cdot n&=\sum_{i,j}n_in_j\partial_i \om_j \\
&=\sum_{i,j}n_i(n_j\partial_i -n_i\partial_j)\om_j  \\
&=\sum_{i,j}(n_j\partial_i -n_i\partial_j)(n_i\om_j) -\sum_{i,j}\om_j(n_j\partial_i -n_i\partial_j)n_i\\
&=\frac12\sum_{i,j}(n_j\partial_i -n_i\partial_j)(n_i\om_j-n_j\om_i)
-\om\cdot n\dive n+\sum_{i,j}\om_jn_i\partial_jn_i\\
&=(n\times\nabla)\cdot(\om\times n)-\om\cdot n\dive n+\frac12\om\cdot\nabla(\|n\|^2)
\end{align*}
Using \eqref{ident0} and the identity $\om\cdot n=(n\times \nabla)\cdot u$ and recalling that $\|n\|^2=1$ in the neighborhood of the boundary completes the proof of the lemma.
\end{proof}

\section{Conormal Sobolev spaces}
\label{conormal}

The conormal Sobolev spaces are defined by using a family of  generator tangent vector fields. Here, in order to simplify the presentation we will use a particular family of generator tangent vector fields. We define it in the following way. Let $U_0=\{x\in\overline\Om\ ;\ d(x,\partial\Om)<\delta\}$ and $U_1=\{x\in\overline\Om\ ;\ d(x,\partial\Om)>\delta/2\}$ and $\varphi_0,\varphi_1\in C^\infty_0(\overline\Om)$ be a partition of unity subordinated to the open cover of $\overline\Om$ given by $\overline\Om=U_0\cup U_1$. We have that $\varphi_0$ is compactly supported in $U_0$ and is equal to 1 in $\Om_{\delta/2}$. The function $\varphi_1$ is compactly supported in $U_1$ and is equal to 1 in $\Om_\delta^c$. Since $\|n\|=1$ on $U_0$, the set
\begin{align*}
\cz&=\Bigl\{
\varphi_0 \begin{pmatrix}0\\-n_3\\n_2\end{pmatrix},
\varphi_0 \begin{pmatrix}n_3\\0\\-n_1\end{pmatrix},
\varphi_0 \begin{pmatrix}-n_2\\n_1\\0\end{pmatrix},
\varphi_0 nd(x,\partial\Omega),
\varphi_1 \begin{pmatrix}1\\0\\0\end{pmatrix},
\varphi_1 \begin{pmatrix}0\\1\\0\end{pmatrix},
\varphi_1 \begin{pmatrix}0\\0\\1\end{pmatrix}
\Bigr\}\\
&\equiv\{Z_1,\dots,Z_7\}
\end{align*}
is clearly a  family of generator tangent vector fields.

If $\be\in \N^7$ is a multi-index, we introduce the notation $\Z^\be=\ZZ1^{\be_1}\dots\ZZ 7^{\be_7}$.  For $m\in\N$, we introduce the so-called conormal Sobolev space $H^m_{co}$ which consists of all square-integrable functions $f$ such that $\Z^\be f\in L^2(\Om)$ for all $|\be|\leq m$. The norm on $H^m_{co}$ is given by
\begin{equation*}
\|f\|_{H^m_{co}}^2=\sum_{|\be|\leq m}\nl2{\Z^\be f}^2.
\end{equation*}
We define in a similar manner $W^{m,\infty}_{co}$ by using the $L^\infty$ norm instead of the $L^2$ norm. Finally, let $X^m$ be defined by
\begin{equation*}
X^m=\{f\ ;\  f\in H^m_{co} \text{ and } \nabla f\in H^{m-1}_{co}\}.
\end{equation*}
with norm
\begin{equation*}
  \|f\|_{X^m}=\|f\|_{H^m_{co}}+\|\nabla f\|_{H^{m-1}_{co}}.
\end{equation*}

It can be checked that the following identity holds true
\begin{equation*}
n\times(n\times u)=
(n_3u_2-n_2u_3) \begin{pmatrix}0\\-n_3\\n_2\end{pmatrix}
+(n_1u_3-n_3u_1)\begin{pmatrix}n_3\\0\\-n_1\end{pmatrix}
+(n_2u_1-n_1u_2) \begin{pmatrix}-n_2\\n_1\\0\end{pmatrix}
\end{equation*}
for any vector field $u$. So, in view of our definition of $\cz$, we have that
\begin{equation*}
  \varphi_0 n\times(n\times u)=
(n_3u_2-n_2u_3) Z_1
+(n_1u_3-n_3u_1)Z_2
+(n_2u_1-n_1u_2) Z_3.
\end{equation*}
Next, because of the identity $\|n\|^2u=-n\times(n\times u)+n(n\cdot u)$ and since on the support of $\varphi_0$ we have that $\|n\|=1$, we can decompose
\begin{align*}
\varphi_0 u
&=-\varphi_0 n\times(n\times u)+\varphi_0 n(n\cdot u)\\
&= (n_2u_3-n_3u_2) Z_1 +(n_3u_1-n_1u_3)Z_2+(n_1u_2-n_2u_1) Z_3+\frac{u\cdot n}{d}Z_4.
\end{align*}

We also trivially have that $\varphi_1u=u_1Z_5+u_2Z_6+u_3Z_7$ and since $\varphi_0+\varphi_1=1$ we finally deduce that the following decomposition holds true for any vector field $u$:
\begin{equation}\label{Z}
\begin{aligned}
u&=\varphi_0 u+\varphi_1 u\\
&= (n_2u_3-n_3u_2) Z_1 +(n_3u_1-n_1u_3)Z_2+(n_1u_2-n_2u_1) Z_3  +\frac{u\cdot n}{d}Z_4 +u_1Z_5+u_2Z_6+u_3Z_7\\
&\equiv\sum_{i=1}^7\tiu_iZ_i.
\end{aligned}
\end{equation}

A very important property of this ``canonical decomposition'' associated to the set $\cz$ of generator vector fields is stated in the following lemma.
\begin{lemma}\label{lemadecomp}
Let $u$ be a divergence free vector field tangent to the boundary. For any $m\in\N$ there exists a constant $C=C(m,\Om)$ such that $\hco m{\tiu_i}\leq C\hco {m+1}u$ and $\uco m{\tiu_i}\leq C\uco {m+1}u$ for every $i\in\{1,\dots,7\}$.
\end{lemma}
\begin{proof}
From the explicit formulas for the $\tiu_i$ the assertion is obvious except for $\tiu_4$. Because $u$ is tangent to the boundary, we can apply Lemma \ref{clar} below to $u\cdot n$ to deduce that
\begin{equation*}
\hco m{\tiu_4}=\bighco m{\frac{u\cdot n}{d}}
\leq C(\hco m{u\cdot n}+\hco m{\partial_n (u\cdot n)}).
\end{equation*}
We have that
\begin{equation}\label{interm}
\begin{aligned}
 \partial_n (u\cdot n)
&=\sum_{i,j}n_i\partial_i(n_j u_j)\\
&=\sum_{i,j}n_i\partial_in_j u_j+\sum_{i,j}n_in_j\partial_iu_j\\
&=\partial_n  n\cdot u+\sum_{i,j}n_i(n_j\partial_i-n_i\partial_j)u_j+\|n\|^2\dive u\\
&=\partial_n  n\cdot u+\sum_{i,j}n_i(n_j\partial_i-n_i\partial_j)u_j.
\end{aligned}
\end{equation}
Because $n_j\partial_i-n_i\partial_j$ are tangential derivatives, we immediately deduce that
\begin{equation*}
\hco m{\partial_n (u\cdot n)}\leq C\hco{m+1}u
\end{equation*}
so
\begin{equation*}
\hco m{\tiu_4}\leq C\hco{m+1}u.
\end{equation*}
A similar argument works for the $W^{m,\infty}_{co}$ spaces so the proof is completed.
\end{proof}

We show now the following easy lemma who was used in the proof of the previous lemma.
\begin{lemma}\label{clar}
Let $f$ be a function vanishing on the boundary of $\Omega$.  For each $m\in\N$ there exists a constant $C=C(m,\Om)$ such that
\begin{equation*}
\bighco m{\frac{f}{d}}\leq C(\hco mf+C\hco m{\partial_n  f})
\end{equation*}
and
\begin{equation*}
\biguco m{\frac{f}{d}}\leq C(\uco mf+C\uco m{\partial_n  f}).
\end{equation*}
\end{lemma}
\begin{proof}
The inequalities are obvious in a compact subset of  $\Omega$ because in such a region  $d$ has a strictly positive uniform lower bound. We only need to prove something in the neighborhood of the boundary. Using local changes of coordinates combined with a partition of unity of the neighborhood of the boundary and recalling that the conormal spaces are invariant by changes of variables, we see that it suffices to prove the stated inequalities in the following setting:
\begin{itemize}
\item $\Om$ is the upper-half of the unit ball $B_+=\{x\in\R^3\ ; \ \|x\|< 1\text{ and }x_3>0\}$.
\item $f$ vanishes on the flat part of $B_+$: $f(x_1,x_2,0)=0$.
\item the conormal spaces are constructed using the vector fields $\partial_1$, $\partial_2$ and $x_3\partial_3$.
\end{itemize}

So we need to prove that
\begin{equation*}
\hco m{f/x_3}\leq C(\hco mf+C\hco m{\partial_3f})
\quad\text{and}\quad
\uco m{f/x_3}\leq C(\uco mf+C\uco m{\partial_3f}).
\end{equation*}
These bounds are easy to prove since we can write by the Taylor formula
\begin{equation*}
\frac{f}{x_3}= \int_0^1\partial_3f(x_1,x_2,tx_3)\,dt
\end{equation*}
so
\begin{equation*}
\partial_1^{\be_1}\partial_2^{\be_2}(x_3\partial_3)^{\be_3}(f/x_3)
= \int_0^1\bigl(\partial_1^{\be_1}\partial_2^{\be_2}(tx_3\partial_3)^{\be_3}\partial_3f\bigr)(x_1,x_2,tx_3)\,dt
\end{equation*}
Taking the $L^\infty$ norm yields
\begin{equation*}
  \nl\infty{\partial_1^{\be_1}\partial_2^{\be_2}(x_3\partial_3)^{\be_3}(f/x_3)}\leq \nl\infty{\partial_1^{\be_1}\partial_2^{\be_2}(x_3\partial_3)^{\be_3}\partial_3 f}
\end{equation*}
while taking the $L^2$ norm gives
\begin{align*}
\nl2{\partial_1^{\be_1}\partial_2^{\be_2}(x_3\partial_3)^{\be_3}(f/x_3)}
&\leq \int_0^1\|\bigl(\partial_1^{\be_1}\partial_2^{\be_2}(tx_3\partial_3)^{\be_3}\partial_3f\bigr)(x_1,x_2,tx_3)\|_{L^2(dx)}\,dt\\
&= \int_0^1\|\bigl(\partial_1^{\be_1}\partial_2^{\be_2}(y_3\partial_{y_3})^{\be_3}\partial_{y_3}f\bigr)(x_1,x_2,y_3)\|_{L^2(dx_1dx_2dy_3)}\frac1{\sqrt t}\,dt.
\end{align*}
The last $L^2$ norm is not on the full domain $B_+$ (like the other $L^2$ norms). Because of the change of variables $y_3=tx_3$, the domain of integration of the last $L^2$ norm is the subset of $B_+$ formed by the triples $(x_1,x_2,tx_3)$ where $x\in B_+$. Since the $L^2$ norm is taken on a subset of $B_+$, we can bound it by the norm on the full $B_+$ obtaining in the end
\begin{equation*}
 \nl2{\partial_1^{\be_1}\partial_2^{\be_2}(x_3\partial_3)^{\be_3}(f/x_3)}\leq \nl2{\partial_1^{\be_1}\partial_2^{\be_2}(x_3\partial_3)^{\be_3}\partial_3 f}\int_0^1\frac1{\sqrt t}\,dt
= 2\nl2{\partial_1^{\be_1}\partial_2^{\be_2}(x_3\partial_3)^{\be_3}\partial_3 f}.
\end{equation*}
This completes the proof of the lemma.
\end{proof}

The next result  shows that the gradient of a divergence free vector field is controlled by the vorticity and by tangential derivatives only.
\begin{lemma}\label{lemalip}
Let $k\in\N$ and $u$ be a divergence free vector field. There exists a constant $C=C(k,\Om)>0$ such that
\begin{equation*}
\uco k{\nabla u}\leq C(\uco k\om+\uco{k+1}u)
\end{equation*}
where $\om=\curl u$.
\end{lemma}
\begin{proof}
In the interior of $\Om$ the bound is obvious, so we only need to prove it in the neighborhood of the boundary. We will prove it in $\Om_\delta$ where $\|n\|=1$.

Because of the identities
\begin{equation*}
\nabla=-\frac n{\|n\|^2}\times(n\times\nabla)+\frac n{\|n\|^2}(\partial_n )
\end{equation*}
and
\begin{equation*}
u=-\frac n{\|n\|^2}\times(n\times u)+\frac n{\|n\|^2}(n\cdot u)
\end{equation*}
we observe that it suffices to bound $\uco k{\partial_n (n\cdot u)}$ and $\uco k{\partial_n (n\times u)}$. Thanks to \eqref{interm} we have that
\begin{equation*}
\uco k{\partial_n (n\cdot u)}  \leq C\uco {k+1} u.
\end{equation*}

To bound $\uco k{\partial_n (n\times u)}$, let us consider for example the first component:
\begin{align*}
\bigl[\partial_n (n\times u)\bigr]_1
&= \sum_i n_i\partial_i(n_2u_3-n_3u_2)\\
&= \sum_i n_i(\partial_in_2u_3-\partial_in_3u_2)+ \sum_i n_i(n_2\partial_iu_3-n_3\partial_iu_2)\\
&= (\partial_n n\times u)_1+ \sum_i n_i[(n_2\partial_i-n_i\partial_2)u_3-(n_3\partial_i-n_i\partial_3)u_2]+\|n\|^2\om_1.
\end{align*}
We infer that
\begin{equation*}
 \uco k{\partial_n (n\times u)}\leq  C(\uco k \om+\uco{k+1} u)
\end{equation*}
and this completes the proof.
\end{proof}

We end this section with the following technical results about the conormal Sobolev spaces:
\begin{lemma}\label{gagliardo}
\begin{enumerate}
\item For all $k\in\N$ and $|\be_1|+|\be_2|\leq k$ we have that
  \begin{gather}
\nl2{\partial_Z^{\be_1}f\partial_Z^{\be_2}g}\leq  C(\nl\infty f\|g\|_{H^k_{co}}+\hco kf\nl\infty g)\label{gn1}\\
\intertext{and}
\hco k{fg}\leq  C(\nl\infty f\|g\|_{H^k_{co}}+\hco kf\nl\infty g).\label{gn2}
  \end{gather}
\item \label{imbedding} The imbedding $X^2\subset L^\infty$ holds true.

\end{enumerate}
\end{lemma}
\begin{proof}
Relation \eqref{gn1}  was proved in \cite[Lemma 8]{masmoudi_uniform_2012-1}. Relation \eqref{gn2} follows from \eqref{gn1} and the Leibniz formula.

We prove now the embedding stated in item \ref{imbedding}). In the interior of $\Omega$ the $H^m_{co}$ regularity is the same as the $H^m$ regularity. Since in dimension three we have the embedding $H^2\subset L^\infty$ the desired embedding holds true in a compact region of $\Om$.  Therefore, we can assume that we are in the neighborhood of the boundary. Using a change of coordinates and a partition of unity, we can assume that the domain $\Om$ is the half-plane $\Omega=\{x\ ; x_3>0\}$. Let us denote $x_h=(x_1,x_2)$ and take some $f\in X^2$. We  have that $f$ and $\nabla_h f\in H^1(\Om)$. By the trace theorem, for all $x_3\geq0$ we have that $f(\cdot,x_3)$ and $\nabla_h f(\cdot,x_3)\in H^{\frac12}(\R^2)$ so $f(\cdot,x_3)\in H^{\frac32}(\R^2)$. The Sobolev embedding $ H^{\frac32}(\R^2)\subset L^\infty(\R^2)$ completes the proof of item \ref{imbedding}).
\end{proof}

\section{Some ellipticity results in conormal spaces}

We start with the following easy lemma relating velocity to vorticity in conormal spaces.
\begin{lemma}\label{lemmareg}
Let $u$ be a divergence free vector field tangent to the boundary.
There exists a constant $B_m=B(m,\Om)$ such that the following inequality holds true:
\begin{equation*}
\xm{m+1}u\leq B_m(\nl2u+\hcom m\om)
\end{equation*}
where $\om=\curl u$.
\end{lemma}
\begin{proof}
Let $\zm$ be a tangential derivative of order $m$. We use \cite[Proposition 1.4]{foias_remarques_1978} to write
\begin{align*}
  \nl2{\nabla\zm u}
&\leq C(\nl2{\zm u}+\nl2{\curl\zm u}+\nl2{\dive\zm u}+\|n\cdot\zm u\|_{H^{1/2}(\partial\Omega)})\\
&\leq C(\hco  mu+\hco m\om+\nl2{[\curl,\zm] u}+\nl2{[\dive,\zm] u}+\|[n\cdot,\zm] u\|_{H^{1/2}(\partial\Omega)})
\end{align*}
where we used that $u$ is divergence free and tangent to the boundary. Clearly
\begin{equation*}
\nl2{[\curl,\zm] u}\leq C\xm m u
\end{equation*}
and
\begin{equation*}
\nl2{[\dive,\zm] u}\leq C\xm m u\cdot
\end{equation*}
We observe now that $[n\cdot,\zm] u$ is a combination of tangential derivatives of $u$ of order $\leq m-1$. But if  $\partial_Z^{m-1}$ is  a tangential derivative of order $\leq m-1$ then we have that
\begin{equation*}
\|\partial_Z^{m-1} u\|_{H^{1/2}(\partial\Omega)} \leq C \|\partial_Z^{m-1} u\|_{H^1(\Omega)} \leq C\xm mu\cdot
\end{equation*}
We infer from the above relations that the following estimate holds true:
\begin{equation*}
  \xm{m+1}u\leq C(\xm mu+\hco m\om).
\end{equation*}
Clearly one can now iterate the argument and bound  the term $\xm mu$ on the right-hand side. After $m$ iterations we obtain the desired conclusion.
\end{proof}

The main result of this section is  the following elliptic estimate:
\begin{proposition}\label{elliptic}
Let $m\in\N$. Suppose that $u$ is divergence free and verifies the Navier boundary conditions \eqref{navier}.
There exists $\al_0=\al_0(\Om,m)$ and a constant $C>0$ such that for all $0<\al<\al_0$ we have that
\begin{equation*}
\xm{m+1}u\leq C(\nl2u+\hco m{\oma}).
\end{equation*}
\end{proposition}
\begin{proof}
We will in fact show that for  $0<\al<\al_0$ (with $\al_0<1$ small enough to be chosen later) there exists a constant $C$ such that
\begin{equation}\label{basic}
\xm {m+1}u^2+\al\xm {m+1}{\om}^2+\al^2\hco m{\Delta \om}^2\leq C(\nl2u^2+\hco m{\oma}^2).
\end{equation}

We proceed by induction. We consider first the case $m=0$.

\subsection*{Case $m=0$.}
Since $X^1=H^1$ and $H^0_{co}=L^2$, we need to prove that if $0<\al<\al_0$ then
\begin{equation*}
 \nh1u^2+\al\nh1{\om}^2+\al^2\nl2{\Delta \om}^2\leq C(\nl2u^2+\nl2{\oma}^2)
\end{equation*}
for some constant $C=C(\al_0,\Om)$.

Clearly
\begin{equation*}
\begin{split}
 \nl2\oma^2&=\nl2\om^2+\al^2\nl2{\Delta\om}^2-2\al\int_\Om\om\cdot\Delta\om \\
&=\nl2\om^2+\al^2\nl2{\Delta\om}^2+2\al\nl2{\nabla\om}^2-    2\al\int_\bdry\om\cdot\partial_n\om \\
\end{split}
\end{equation*}

We use Lemma \ref{ident} to write the boundary terms under the form:
\begin{align*}
\int_\bdry\om\cdot\partial_n\om
&=\int_\bdry\om\tan\cdot(\partial_n\om)\tan  +\int_\bdry\om\nor(\partial_n\om)\nor  \\
&=\int_\bdry F(u)\cdot(\partial_n\om)\tan  +\int_\bdry\om\nor G(u,(n\times\nabla) u)\\
&\equiv I_1+I_2.
\end{align*}
We go back to an integral on $\Om$ by means of the Stokes formula:
\begin{equation*}
I_2=\int_\bdry\om\nor G(u,(n\times\nabla) u)
=\int_\bdry\|n\|^2\om\nor G(u,(n\times\nabla) u)
=\int_\Om\sum_i\partial_i[n_i\om\nor G(u,(n\times\nabla) u)]
\end{equation*}
so
\begin{equation*}
|I_2|\leq C(\nl2\om\nh2u+\nh1\om\nh1u).
\end{equation*}

We use again the Stokes formula to write
\begin{multline*}
I_1=\int_\bdry F(u)\cdot(\partial_n\om)\tan
=  \int_\bdry F(u)\cdot  (\sum_i n_i \partial_i\om)\tan
=  \int_\bdry \sum_i n_i F(u)\cdot (\partial_i\om)\tan\\
=  \int_\Om \sum_i \partial_i [F(u)\cdot (\partial_i\om)\tan].
\end{multline*}
Expanding the last term above and separating the terms containing second order derivatives of $\om$, we observe that we can bound pointwise
\begin{equation*}
 |\sum_i \partial_i [F(u)\cdot (\partial_i\om)\tan]- F(u)\cdot (\Delta\om)\tan|
\leq C(|u|+|\nabla u|)|\nabla\om|.
\end{equation*}
We infer that
\begin{equation*}
|I_1|\leq C\int_\Om(|u|+|\nabla u|)|\nabla\om|+C\int_\Om|F(u)\cdot (\Delta\om)\tan|
\leq C\nh1u\nl2{\nabla\om}+C\nl2u\nl2{\Delta\om}.
\end{equation*}

The previous relations imply that
\begin{equation*}
  \bigl|\int_\bdry\om\cdot\partial_n\om\bigr|\leq  C\nl2\om\nh2u
+C\nh1{ \om}\nh1u+ C\nl2u\nl2{\Delta\om}
\end{equation*}

But we have that
$\nl2u+\nl2\om\simeq\nh1u$ and $\nl2u+\nh1\om\simeq\nh2u$ (see \cite[Proposition 1.4]{foias_remarques_1978}), so we can further write that
\begin{equation*}
  \bigl|\int_\bdry\om\cdot\partial_n\om\bigr|\leq  C(\nl2u+\nl2\om)(\nl2\om+\nl2{\nabla\om})+C\nl2u\nl2{\Delta\om}
\end{equation*}

We conclude that
\begin{align*}
 \nl2\oma^2
&\geq \nl2\om^2+\al^2\nl2{\Delta\om}^2+2\al\nl2{\nabla\om}^2
-  C\al\nl2u\nl2{\Delta\om} \\
&\hskip 5cm -C\al(\nl2u+\nl2\om)(\nl2\om+\nl2{\nabla\om})\\
&\geq (1-C\al)\nl2\om^2+\frac{\al^2}2\nl2{\Delta\om}^2+\al\nl2{\nabla\om}^2 -C\nl2u^2.
\end{align*}

We finally obtain that
\begin{align*}
\nl2u^2+\ep_0   \nl2\oma^2
&\geq (1-C\ep_0)\nl2u^2+ \ep_0(1-C\al)\nl2\om^2+\frac{\ep_0\al^2}2\nl2{\Delta\om}^2+\ep_0\al\nl2{\nabla\om}^2\\
&\geq C(\ep_0,\al)(\nh1u^2+\al^2\nl2{\Delta\om}^2+\al\nl2{\nabla\om}^2)
\end{align*}
provided that $\al$ and $\ep_0$ are sufficiently small.
This completes the proof in the case $m=0$.

We show now that step $m-1$ implies step $m$.

\subsection*{Step $m-1$ implies step $m$}
We assume that we have proved
\begin{equation}\label{basic1}
\xm {m}u^2+\al\xm m{\om}^2+\al^2\hco {m-1}{\Delta \om}^2\leq K_{m-1}(\nl2u^2+\hco {m-1}{\oma}^2)
\end{equation}
for some constant $K_{m-1}$ and we want to prove that
\begin{equation}\label{basic2}
\xm {m+1}u^2+\al\xm{m+1}{\om}^2+\al^2\hco m{\Delta \om}^2\leq K_m(\nl2u^2+\hco m{\oma}^2)
\end{equation}
for some other constant $K_m$.

Let $\zm=\partial_Z^\be$ be a tangential derivative of order less than $m$:  $\be\in\N^7$ verifies $|\be|\leq m$.

If $\partial_W$ is a tangential derivative, we will denote by $\partial_W^t$ the transpose of $\partial_W$, \textit{i.e.} if $\partial_W=\sum_i W_i\partial_i$ then $\partial_W^t f=-\sum_i\partial_i(W_i f)=-\dive W f-\partial_W f$. Because  $\partial_W$ is a tangential derivative, we have that $\int_\Om\partial_Wf g=\int_\Om f\partial_W^t g$ for all $f$ and $g$ without need to assume any boundary conditions on $f$ and $g$.

We have that
\begin{equation*}
 \nl2{\zm\oma}^2 = \nl2{\zm\om}^2+\al^2\nl2{\zm\Delta\om}^2-2\al\int_\Om \zm\om\cdot \zm\Delta\om
\end{equation*}
We perform now several integrations by parts:
\begin{equation}\label{somebound}
\begin{aligned}
 -\int_\Om \zm\om\cdot \zm\Delta\om
&=-\int_\Om (\zm)^t\zm\om\cdot \Delta\om \\
&=\int_\Om \nabla (\zm)^t\zm\om\cdot \nabla\om -\int_\bdry (\zm)^t\zm\om\cdot \partial_n\om
\end{aligned}
\end{equation}
We wish now to commute the gradient with $\zm$. Repeatedly using the formula
\begin{equation}\label{repeat}
\int_\Om\partial_i\partial_W^tf\;g
=\int_\Om \partial_if\;\partial_Wg
-\int_\Om fg\;\partial_i\dive w
-\sum_j\int_\Om\partial_j f\,g\;\partial_iw_j
\end{equation}
we observe that we can write
\begin{equation*}
\int_\Om \nabla (\zm)^t\zm\om\cdot \nabla\om
=  \int_\Om \nabla \zm\om\cdot\zm \nabla\om +I_1
\end{equation*}
where
\begin{equation*}
|I_1|\leq C\xm{m+1}\om \xm m\om .
\end{equation*}
Moreover,
\begin{equation*}
 \int_\Om \nabla \zm\om\cdot\zm \nabla\om
=\frac12 \nl2{\nabla \zm\om}^2+\frac12 \nl2{\zm\nabla \om}^2-\frac12\nl2{[\nabla, \zm]\om}^2
\end{equation*}
where the last term can be bounded by
\begin{equation*}
 \nl2{[\nabla, \zm]\om}\leq C\xm m\om^2.
\end{equation*}
It remains to estimate the boundary term in \eqref{somebound}.
To do that, we proceed as in the case $m=0$ by decomposing $\om=\om\tan+\om\nor$ and writing
\begin{align*}
\int_\bdry (\zm)^t\zm\om\cdot \partial_n\om
&=\int_\bdry [(\zm)^t\zm\om]\tan\cdot (\partial_n\om)\tan +\int_\bdry [(\zm)^t\zm\om]\nor (\partial_n\om)\nor   \\
&\equiv J_1+J_2
\end{align*}

Using Lemma \ref{ident} and the Stokes formula we can write
\begin{align*}
J_2&= \int_\bdry n\cdot (\zm)^t\zm\om \  G(u,(n\times\nabla) u) \\
&=\sum_i \int_\Om \partial_i\bigl[(\zm)^t\zm\om_i\  G(u,(n\times\nabla) u)\bigr]\\
&=\sum_i \int_\Om \partial_i\bigl[(\zm)^t\zm\om_i\bigr]\  G(u,(n\times\nabla) u)
+\sum_i \int_\Om (\zm)^t\zm\om_i\  \partial_i\bigl[G(u,(n\times\nabla) u)\bigr]\\
&=\sum_i \int_\Om \partial_i\bigl[(\zm)^t\zm\om_i\bigr]\  G(u,(n\times\nabla) u)
+\sum_i \int_\Om \partial_Z^{m+1}\om_i\  \partial_Z^{m-1}\partial_i\bigl[G(u,(n\times\nabla) u)\bigr]\\
&\equiv J_{21}+J_{22}
\end{align*}
where $\Z^{m+1}$ denotes a tangential derivative of order $\leq m+1$ and $\Z^{m-1}$ denotes a tangential derivative of order $\leq m-1$.
Clearly
\begin{equation*}
|J_{22}|\leq C\nl2{\partial_Z^{m+1}\om}\nl2{\partial_Z^{m-1}\nabla[G(u,(n\times\nabla) u)]}
\leq C\hco{m+1}\om  \xm {m+1}u
\end{equation*}
Repeatedly using relation \eqref{repeat} we can also bound
\begin{equation*}
|J_{21}|\leq C \xm{m+1}\om  \hco{m+1}u
\end{equation*}

\medskip

We go now to the estimate of the term $J_1$. Recalling that in the neighborhood of the boundary we have the decomposition $\om=n\times \om\tan+\om\nor n$, we can write
\begin{multline*}
J_1= \int_\bdry [(\zm)^t\zm\om]\tan\cdot (\partial_n\om)\tan
=\int_\bdry [(\zm)^t\zm(n\times\om\tan+\om\nor n)]\tan\cdot (\partial_n\om)\tan \\
=\int_\bdry [(\zm)^t\zm(n\times \om\tan)]\tan\cdot (\partial_n\om)\tan
+\int_\bdry [(\zm)^t\zm(\om\nor n)]\tan\cdot (\partial_n\om)\tan
\equiv J_{11}+J_{12}.
\end{multline*}
Using Lemma \ref{ident},  the fact that $\Z$ is a tangential derivative and that $\Z^t$ is $-\Z$ plus a zero order term, we deduce that $(\zm)^t\zm(n\times\om\tan)=(\zm)^t\zm (n\times F(u))=-(\zm)^t\zm F(u)\tan$ on the boundary. We infer that
\begin{align*}
-J_{11}
&= \sum_i \int_\bdry n_i[(\zm)^t\zm F(u)\tan]\tan \cdot(\partial_i\om)\tan \\
&= \sum_i \int_\Om \partial_i\left\{[(\zm)^t\zm F(u)\tan]\tan \cdot(\partial_i\om)\tan \right\}\\
&=\sum_i \int_\Om \partial_i\left\{[(\zm)^t\zm F(u)\tan]\tan \right\} \cdot(\partial_i\om)\tan
+\sum_i \int_\Om [(\zm)^t\zm F(u)\tan]\tan \cdot\partial_i[(\partial_i\om)\tan ]\\
&=\sum_i \int_\Om \partial_i\left\{[(\zm)^t\zm F(u)\tan]\tan \right\} \cdot(\partial_i\om)\tan
+\sum_i \int_\Om [(\zm)^t\zm F(u)\tan]\tan \cdot(\Delta\om)\tan \\
&\hskip 7cm +\sum_i \int_\Om [(\zm)^t\zm F(u)\tan]\tan \cdot(\partial_i\om\times \partial_i n)\\
&\equiv J_{111}+J_{112}+J_{113}\cdot
\end{align*}
Using relation \eqref{repeat}  $m$ times we can bound
\begin{equation*}
|J_{111}|\leq C  \xm {m+1}u \xm {m+1}\om
\end{equation*}
Integrating by parts $m$ times allows to estimate
\begin{equation*}
|J_{112}|\leq C  \hco {m}u \hco {m}{\Delta\om}
\end{equation*}
and
\begin{equation*}
|J_{113}|\leq C  \hco {m}u \xm {m+1}{\om}.
\end{equation*}

This completes the estimate of the term $J_{11}$. We claim that exactly the same estimates hold true for the term $J_{12}$. Indeed, the key point that allowed us to estimate $J_{11}$ is the fact that thanks to Lemma \ref{ident}, on the boundary the expression $[(\zm)^t\zm(n\times\om\tan)]\tan$ can be written as a combination of tangential derivatives of $u$ of order $2m$ at most. But exactly the same holds true for  the expression $[(\zm)^t\zm(\om\nor n)]\tan$. Indeed, because of the identity $\om\nor=\om\cdot n=(n\times \nabla)\cdot u$ and thanks to the Leibniz formula, we can write
\begin{equation*}
 [(\zm)^t\zm(\om\nor n)]\tan=  (\zm)^t\zm((n\times \nabla)\cdot u\, n)\times n
= (\zm)^t\zm((n\times \nabla)\cdot u)\, n\times n +\Gamma
\end{equation*}
where the expression $\Gamma$ is a linear combination of tangential derivatives of $u$ of order $2m$ at most. The first term on the right-hand side vanishes, so we can conclude that the estimates we proved for $J_{11}$ hold true for $J_{12}$ as well.

\medskip

From the previous estimates we infer that
\begin{multline*}
\nl2 {\zm\oma}^2\geq  \nl2{\zm\om}^2+\al^2\nl2{\zm\Delta\om}^2+\al \nl2{\nabla \zm\om}^2+\al\nl2{\zm\nabla \om}^2\\
-C\al(\xm m\om\xm{m+1}\om+\hco{m+1}\om\xm{m+1}u+\xm{m+1}\om\hco{m+1}u\\
+\xm{m+1}u\xm {m+1}\om+\hco mu\hco m{\Delta\om})
\end{multline*}

Summing over all possible choices of $\zm$ we get
\begin{equation}\label{fx}
\begin{aligned}
\hcom m {\oma}^2
& \geq \hcom m {\om}^2+\al^2\hcom m{\Delta\om}^2+\al\hcom m {\nabla \om}^2
-C\al(\xm m\om\xm{m+1}\om+\hco{m+1}\om\xm{m+1}u\\
&\hskip 4cm +\xm{m+1}\om\hco{m+1}u+\xm{m+1}u\xm {m+1}\om+\hco mu\hco m{\Delta\om})\\
&=  \hcom m {\om}^2+\al^2\hcom m{\Delta\om}^2+\al\hcom m {\nabla \om}^2
-C\al R
\end{aligned}
\end{equation}
where
\begin{equation*}
R=  \xm m\om\xm{m+1}\om+\hco{m+1}\om\xm{m+1}u+\xm{m+1}\om\hco{m+1}u
+\xm{m+1}u\xm {m+1}\om+\hco mu\hco m{\Delta\om} .
\end{equation*}

To prove \eqref{basic2} it clearly suffices to show that there exists $\ep>0$ and $K'_m$ such that
\begin{equation}\label{basic2a}
\xm {m+1}u^2+\al\xm{m+1}{\om}^2+\al^2\hco m{\Delta \om}^2\leq K'_m(\nl2u^2+\hco {m-1}{\oma}^2+\ep\hcom m{\oma}^2)
\end{equation}
Using \eqref{basic1} and \eqref{fx} we have that
\begin{multline*}
\nl2u^2+\hco {m-1}{\oma}^2+\ep\hcom m{\oma}^2
\geq \frac1{K_{m-1}} (\xm {m}u^2+\al\xm m{\om}^2+\al^2\hco {m-1}{\Delta \om}^2  ) \\
+\ep\hcom m {\om}^2+\ep\al^2\hcom m{\Delta\om}^2+\ep\al\hcom m {\nabla \om}^2
-C\al\ep R
\end{multline*}

Thanks to Lemma \ref{lemmareg} we can estimate
\begin{equation*}
  \frac1{K_{m-1}} \xm {m}u^2 +\ep\hcom m {\om}^2=  \frac1{2K_{m-1}} \xm {m}u^2 + \frac1{2K_{m-1}} \xm {m}u^2 +\ep\hcom m {\om}^2
\geq  \frac1{2K_{m-1}} \xm {m}u^2+\frac\ep{B^2_m} \xm {m+1}u^2
\end{equation*}
provided that $\ep\leq \frac1{2K_{m-1}}$ which we will assume to hold true in what follows. Writing also
\begin{equation*}
\frac1{K_{m-1}} \xm m\om^2 + \ep \hcom m{\nabla \om}^2
\geq \frac1{2K_{m-1}} \xm m\om^2 +C_1 \ep \xm{m+1}\om^2
\end{equation*}
we infer from the above relations that
\begin{multline}\label{final1}
 \nl2u^2+\hco {m-1}{\oma}^2+\ep\hcom m{\oma}^2
\geq  \frac1{2K_{m-1}} (\xm {m}u^2+\al\xm m{\om}^2+\al^2\hco {m-1}{\Delta \om}^2  )\\
+C_2\ep (\xm {m+1}u^2+\al\xm {m+1}{\om}^2+\al^2\hco {m}{\Delta \om}^2 )-C\al\ep R.
\end{multline}

It remains to estimate the term $C\al\ep R$. We bound first
\begin{align*}
R&=  \xm m\om\xm{m+1}\om+\hco{m+1}\om\xm{m+1}u+\xm{m+1}\om\hco{m+1}u
+\xm{m+1}u\xm {m+1}\om+\hco mu\hco m{\Delta\om}  \\
&\leq C(\xm m\om\xm{m+1}\om+\xm{m+1}\om\xm{m+1}u+\xm mu\hco m{\Delta\om})
\end{align*}
We use next Lemma \ref{lemmareg} to write $\xm {m+1}u\leq C(\xm mu+\xm m\om)$ and deduce that
\begin{align*}
C\al\ep R&\leq  C\al\ep \xm{m+1}\om(\xm m\om+\xm{m}u)+ C\al\ep \xm mu\hco m{\Delta\om} \\
&\leq \frac{C_2\ep}{2} (\al\xm {m+1}{\om}^2+\al^2\hco {m}{\Delta \om}^2) +C\ep(1+\al) \xm mu^2+C\al\ep\xm m\om^2.
\end{align*}
Using this bound in \eqref{final1} implies that
\begin{align*}
 \nl2u^2+\hco {m-1}{\oma}^2+\ep\hcom m{\oma}^2
&\geq \frac{C_2\ep}2 (\xm {m+1}u^2+\al\xm {m+1}{\om}^2+\al^2\hco {m}{\Delta \om}^2 ) \\
&\hskip 1cm  +\bigl(\frac1{2K_{m-1}}-C\ep(1+\al)\bigr) \xm mu^2
+\al\bigl(\frac1{2K_{m-1}}-C\ep\bigr) \xm m\om^2\\
& \geq \frac{C_2\ep}2 (\xm {m+1}u^2+\al\xm {m+1}{\om}^2+\al^2\hco {m}{\Delta \om}^2 )
\end{align*}
provided that  $\ep$ is sufficiently small. The above relation implies that  \eqref{basic2a} holds true. This completes the proof.
\end{proof}

We will also need some $W^{1,\infty}_{co}$ elliptic estimates for the operator $1-\al\Delta$ in the setting of the conormal Sobolev spaces. We start with an $L^\infty$ bound.
\begin{lemma}\label{maxprinc}
There exists a constant $C$ independent of $\al$ such that the following relation holds true:
\begin{equation*}
\nl\infty h+\sqrt\al\nl\infty{\nabla h}+\al\nl\infty{\Delta h}\leq C(\nl\infty{h-\al\Delta h}+  \|h\|_{L^{\infty}(\partial\Omega)}+\sqrt\al \|h\|_{W^{1,\infty}(\partial\Omega)} +\al\|h\|_{W^{2,\infty}(\partial\Omega)}).
\end{equation*}
\end{lemma}
\begin{proof}
We assume first that $h$ vanishes on the boundary of $\Om$. In this case, it was proved in \cite[Lemma A.2]{bethuel_asymptotics_1993} the following inequality:
\begin{equation*}
\nl\infty{\nabla h}^2\leq  C_1\nl\infty h\nl\infty {\Delta h}.
\end{equation*}
From the maximum principle we have that
\begin{equation*}
 \nl\infty h\leq \nl\infty {h-\al\Delta h}
\end{equation*}
so
\begin{equation*}
\nl\infty {\Delta h}= \frac1\al\nl\infty {h-\al\Delta h-h}
\leq \frac1\al(\nl\infty {h-\al\Delta h}+\nl\infty {h})
\leq \frac2\al\nl\infty {h-\al\Delta h}.
\end{equation*}
We conclude that
\begin{equation}\label{bbh}
 \al\nl\infty{\nabla h}^2\leq  2C_1\nl\infty h\nl\infty {h-\al\Delta h} \leq  2C_1\nl\infty {h-\al\Delta h}^2
\end{equation}
which completes the proof in the case when $h$ vanishes on the boundary.

We consider now the general case. Let $H$ be a $W^{2,\infty}$  extension of $h\bigl|_{\partial\Om}$ to $\Om$ such that $\|H\|_{W^{k,\infty}(\Omega)}\leq C \|h\|_{W^{k,\infty}(\partial\Omega)}$ for all $k\in\{0,1,2\}$, where $C$ depends only on $\Om$. Because $h-H$ vanishes on the boundary, we can apply relation \eqref{bbh} to $h-H$ to obtain:
\begin{multline*}
 \sqrt\al\nl\infty{\nabla (h-H)}\leq  C\nl\infty {h-H-\al\Delta (h-H)}
\leq  C\nl\infty {h-\al\Delta h} +C\nl\infty H+ C\al\|H\|_{W^{2,\infty}} \\
\leq   C\nl\infty {h-\al\Delta h} +C\|h\|_{L^{\infty}(\partial\Om)}+ C\al\|h\|_{W^{2,\infty}(\partial\Om)}.
\end{multline*}
We infer that
\begin{align*}
 \sqrt\al\nl\infty{\nabla h}
&\leq  \sqrt\al\nl\infty{\nabla H}+ C\nl\infty {h-\al\Delta h} +C\|h\|_{L^{\infty}(\partial\Om)}+ C\al\|h\|_{W^{2,\infty}(\partial\Om)}\\
&\leq C\nl\infty {h-\al\Delta h} +C\|h\|_{L^{\infty}(\partial\Om)}+ C\sqrt\al\|h\|_{W^{1,\infty}(\partial\Om)} +C\al\|h\|_{W^{2,\infty}(\partial\Om)}.
\end{align*}

The $L^\infty$ bound for $h$ follows from the maximum principle and the $L^\infty$ bound for $\Delta h$ is obvious from the triangle inequality: $\al\nl\infty{\Delta h}\leq \nl\infty h+\nl\infty{h-\al\Delta h}$. This completes the proof.
\end{proof}

We can now prove the $W^{1,\infty}_{co}$ estimates  for  $1-\al\Delta$.
\begin{lemma}\label{lemalipco}
Suppose that $u$ is divergence free and verifies the Navier boundary conditions \eqref{navier}.
There exists $\al_0=\al_0(\Om)$ and a constant $C=C(\Om)>0$ such that for all $0<\al<\al_0$ we have that
\begin{equation*}
\lipco\om\leq C(\lipco\oma+ \uco2u+\sqrt\al\uco3u+\al\uco4u).
\end{equation*}
\end{lemma}
\begin{proof}
Recall that
\begin{equation}\label{eqom}
\om-\al\Delta\om=\oma
\end{equation}
and that
\begin{equation*}
\om\cdot n=(n\times\nabla)\cdot u.
\end{equation*}
Because of the identity $\om\|n\|^2=n(\om\cdot n)-(\om\times n)\times n$ and using relation \eqref{ident0} we observe that
\begin{equation*}
\om=n [(n\times\nabla)\cdot u]-F(u)\times n\quad\text{on }\partial\Om.
\end{equation*}
Therefore, for $k\in\N$, we have the bound
\begin{equation}\label{boundombord}
 \|\om\|_{W^{k,\infty}(\partial\Omega)}\leq C\uco{k+1}u.
\end{equation}

We use  Lemma \ref{maxprinc} to deduce that
\begin{equation}\label{Y}
\begin{aligned}
 \nl\infty \om+\sqrt\al\|\om\|_{W^{1,\infty}}+\al\nl\infty{\Delta \om}
&\leq C(\nl\infty{\oma}+  \|\om\|_{L^{\infty}(\partial\Omega)}+\sqrt\al \|\om\|_{W^{1,\infty}(\partial\Omega)} \\
&\hskip 7cm +\al\|\om\|_{W^{2,\infty}(\partial\Omega)})\\
&\leq C(\nl\infty{\oma}+ \uco1u+\sqrt\al\uco2u+\al\uco3u).
\end{aligned}
\end{equation}

Next we apply a tangential derivative $\partial_Z$ to \eqref{eqom} and obtain
\begin{equation*}
\partial_Z\om-\al\Delta\partial_Z\om=\partial_z\oma +\al[\partial_Z, \Delta]\om.
\end{equation*}
As above, we deduce from  Lemma \ref{maxprinc} the following inequality:
\begin{multline}\label{X}
 \nl\infty {\Z\om}+\sqrt\al\nl\infty{\nabla \Z\om}
\leq C(\nl\infty{\Z\oma}+\al\nl\infty{[\partial_Z, \Delta]\om}+ \uco2u\\ +\sqrt\al\uco3u+\al\uco4u).
\end{multline}

We prove now that the following estimate holds true:
\begin{equation}\label{secondlip}
\|\om\|_{W^{2,\infty}}  \leq C(\|\om\|_{W^{1,\infty}}+\lipco{\nabla\om}+\nl\infty{\Delta\om}).
\end{equation}
The inequality is obvious in $\Om\setminus\Om_\delta$, so we only need to prove it on $\Om_\delta$. But in this region we have that $\|n\|=1$ so
\begin{equation*}
\nabla=-n\times(n\times\nabla)+n\partial_n.
\end{equation*}
Because $n\times\nabla$ is a tangential derivative, to prove \eqref{secondlip} it suffices to show that
\begin{equation}\label{secondlip2}
\|\partial_n^2\om\|_{W^{2,\infty}(\Om_\delta)}  \leq C(\|\om\|_{W^{1,\infty}}+\lipco{\nabla\om}+\nl\infty{\Delta\om}).
\end{equation}
But
\begin{equation*}
\Delta=\nabla\cdot\nabla= \bigl(n\times(n\times\nabla)-n\partial_n\bigr)\cdot  \bigl(n\times(n\times\nabla)-n\partial_n\bigr)
\end{equation*}
and
\begin{equation*}
 (n\partial_n)\cdot (n\partial_n) =n\cdot\partial_nn\partial_n+\|n\|^2\partial_n^2=\frac12\partial_n(\|n\|^2)\partial_n+\|n\|^2\partial_n^2=\partial_n^2\quad\text{on }\Om_\delta
\end{equation*}
because $\|n\|=1$ on $\Om_\delta$. This observation immediately implies relation \eqref{secondlip2}, so \eqref{secondlip} is proved.

Next, since $[\partial_Z, \Delta]\om$ is a linear combination of derivatives of second order or less of $\om$, we can use relations \eqref{secondlip} and \eqref{Y} to bound
\begin{multline*}
 \al \nl\infty{[\partial_Z, \Delta]\om}\leq C\al\|\om\|_{W^{2,\infty}}\leq C\al(\|\om\|_{W^{1,\infty}}+\lipco{\nabla\om}+\nl\infty{\Delta\om})\\
\leq C\al(\|\om\|_{W^{1,\infty}}+\lipco{\nabla\om})+C(\nl\infty{\oma}+ \uco1u+\sqrt\al\uco2u+\al\uco3u).
\end{multline*}
Using this relation in the bound for $\Z\om$ given in \eqref{X}, adding to the bound for $\om$ given in \eqref{Y} and summing over all tangential derivatives $\Z$ implies
\begin{multline*}
 \lipco{\om}+ \sqrt\al\|\om\|_{W^{1,\infty}}+\sqrt\al \lipco{\nabla\om}
\leq C\al(\|\om\|_{W^{1,\infty}}+\lipco{\nabla\om})\\
 +C(\lipco\oma+ \uco2u+\sqrt\al\uco3u+\al\uco4u).
\end{multline*}
If $\al$ is sufficiently small, the first term on the right-hand side can be absorbed in the left-hand side and the conclusion follows.
\end{proof}

\section{A priori estimates}
In this section, we prove some \textit{a priori} estimates for Theorems \ref{uniformtime} and \ref{euler}. These  \textit{a priori} estimates will be used in conjunction with an approximation procedure to yield the rigorous existence of the solutions in the next section. Let us first briefly
explain why introduce conormal spaces into this problem.

The standard existence result for $\al$-Euler in three space dimensions gives strong solutions in $H^3$, obtained by means of $H^3$ a priori estimates on the velocity. For Navier boundary conditions, one does not expect $H^3$ bounds on the velocity uniformly in $\al$, which would be required to obtain a time of existence uniform in $\al$. Indeed, if such bounds were available, then by the result from \cite{busuioc_incompressible_2012}, we would conclude that solutions of the $\al$-Euler equations converge weakly in $H^3$ to a solution of the Euler equation. But weak convergence in $H^3$ carries the Navier boundary conditions to the limit, so we would find that the corresponding solution of the Euler equation would verify the Navier boundary condition, something which is not true in general.

Something else is needed to obtain a uniform time of existence. Ideally, we would like to prove existence of weak $H^1$ solutions, but even though $H^1$ energy estimates are available, the nonlinearity is too strong to obtain an existence result from such estimates. We propose instead a new type of strong solution, whose regularity involves only one normal derivative and not two or more at the boundary. As we explained above, due to the discrepancy between Navier and non-penetration boundary conditions, we do not expect to be able to control two normal derivatives of the velocity uniformly in $\al$. A similar difficulty is present in the vanishing viscosity limit, and the idea to use conormal spaces to deal with it is originally due to Masmoudi and Rousset, see \cite{masmoudi_uniform_2012-1}.

We begin with the \textit{a priori} estimates required for Theorem \ref{uniformtime}.
\begin{proposition}\label{apriori1}
Let $u$ be a solution of \eqref{maineq} with boundary conditions \eqref{navier} and let $m\geq5$. There exists two constants $\al_0=\al_0(\Om,m)$ and $C=C(\Om,m)$ such that for all $0<\al<\al_0$ the following \textit{a priori} estimates hold true:
\begin{equation*}
L(t)\leq C(\nl2{u_0}+\|\oma_0\|_{H^{m-1}_{co}\cap W^{1,\infty}_{co}})+C\int_0^tL^2(s)\,ds
\end{equation*}
where
\begin{equation*}
L=\|u\|_{X^m\cap W^{1,\infty}}+\|\oma\|_{H^{m-1}_{co}\cap W^{1,\infty}_{co}}.
\end{equation*}
\end{proposition}
\begin{proof}
We start by making $H^{m-1}_{co}$ estimates on the equation verified by the vorticity given in \eqref{eqvort}. We apply $\partial_Z^\be$ to \eqref{eqvort}, multiply by $\partial_Z^\be\oma$, we sum over $|\be|\leq m-1$ and we integrate in space to obtain that
\begin{equation*}
\frac12 \partial_t\hco{m-1}{\oma}^2
=-\sum_{|\be|\leq {m-1}}\int_\Omega \partial_Z^\be(u\cdot\nabla\oma)\partial_Z^\be\oma+ \sum_{|\be|\leq {m-1}}\int_\Omega \partial_Z^\be(\oma\cdot\nabla u)\partial_Z^\be\oma
\equiv I_1+I_2.
\end{equation*}
 We first bound $I_2$ by using Lemma \ref{gagliardo}, item a):
 \begin{equation*}
|I_2|\leq C\| \oma\cdot\nabla u\|_{H^{m-1}_{co}} \| \oma\|_{H^{m-1}_{co}}
\leq C \| \oma\|_{H^{m-1}_{co}} (\| \oma\|_{H^{m-1}_{co}}\nl\infty{\nabla u}+ \| \nabla u\|_{H^{m-1}_{co}}\nl\infty{\oma})
 \end{equation*}

To bound $I_1$, we use the decomposition from relation \eqref{Z}, $u=\sum\limits_{i=1}^7 \tiu_i Z_i$, and write
\begin{align*}
-I_1&= \sum_{|\be|\leq {m-1}}\int_\Omega \partial_Z^\be(u\cdot\nabla\oma)\partial_Z^\be\oma\\
&= \sum_{|\be|\leq {m-1}}\sum_{i=1}^7 \int_\Omega \partial_Z^\be(\tiu_i\partial_{Z_i}\oma)\partial_Z^\be\oma\\
&= \sum_{|\be|\leq {m-1}}\sum_{i=1}^7 \int_\Omega \tiu_i\partial_Z^\be\partial_{Z_i}\oma\partial_Z^\be\oma+I_{11}\\
&= \sum_{|\be|\leq {m-1}}\sum_{i=1}^7 \int_\Omega \tiu_i\partial_{Z_i}\partial_Z^\be\oma\partial_Z^\be\oma
+\sum_{|\be|\leq {m-1}}\sum_{i=1}^7 \int_\Omega \tiu_i[\partial_Z^\be,\partial_{Z_i}]\oma\partial_Z^\be\oma
+I_{11}\\
&= \sum_{|\be|\leq {m-1}}\int_\Omega u\cdot\nabla \partial_Z^\be\oma\partial_Z^\be\oma
+\sum_{|\be|\leq {m-1}}\sum_{i=1}^7 \int_\Omega \tiu_i[\partial_Z^\be,\partial_{Z_i}]\oma\partial_Z^\be\oma
+I_{11}\\
&\equiv I_{12}+I_{13}+I_{11},
\end{align*}
where
\begin{equation*}
I_{11}= \sum_{|\be|\leq {m-1}}\sum_{i=1}^7 \int_\Omega \bigl[\partial_Z^\be(\tiu_i\partial_{Z_i}\oma)-\tiu_i\partial_Z^\be\partial_{Z_i}\oma\bigr]\partial_Z^\be\oma.
\end{equation*}

Now, an integration by parts using that $u$ is divergence free and tangent to the boundary immediately yields that $I_{12}=0$. Next, we observe that $[\partial_Z^\be,\partial_{Z_i}]$ is a combination of tangential derivatives of order $\leq {m-1}$ so we can bound
\begin{equation*}
|I_{13}|\leq  \sum_{|\be|\leq {m-1}}\sum_{i=1}^7 \nl\infty{\tiu_i}\nl2{[\partial_Z^\be,\partial_{Z_i}]\oma}\nl2{\partial_Z^\be\oma}
\leq C\lipco u\hco{m-1}{\oma}^2
\end{equation*}
where we used Lemma \ref{lemadecomp} to bound $\nl\infty{\tiu_i}\leq C\lipco u$.

We estimate now $I_{11}$.
We remark that it can be written as a sum of terms of the form
\begin{equation*}
\int_\Omega \partial_Z^{\gamma_1} \tiu\  \partial_Z^{\gamma_2}\oma \, \partial_Z^\be\oma \quad\text{with }1\leq|\beta|\leq {m-1},\  |\gamma_1|+|\gamma_2|\leq {m}, \ |\gamma_1|,|\gamma_2|\geq1.
\end{equation*}
We now estimate  a term of the form given above. Since $\gamma_1\neq0$ and $\gamma_2\neq0$,  we can write $\partial_Z^{\gamma_1}\tiu=\partial_Z^{\gamma_3}\partial_{Z_j}\tiu$ and $\partial_Z^{\gamma_2}\oma=\partial_Z^{\gamma_4}\partial_{Z_k}\oma$ for some $j$ and $k$. Clearly $|\gamma_3|+|\gamma_4|\leq {m-2}$. Using Lemma \ref{gagliardo}, item a) with $k={m-2}$ and Lemma \ref{lemadecomp} we observe that we can bound
\begin{align*}
\Bigl|\int_\Omega \partial_Z^{\gamma_1} \tiu\  \partial_Z^{\gamma_2}\oma \, \partial_Z^\be\oma\Bigr|
&\leq \nl2{\partial_Z^{\gamma_3}\partial_{Z_j} \tiu\  \partial_Z^{\gamma_4}\partial_{Z_k}\oma}\nl2{\partial_Z^\be\oma}  \\
&\leq C(\nl\infty{\partial_{Z_j} \tiu}\hco{m-2}{\partial_{Z_k}\oma}+\hco{m-2}{\partial_{Z_j} \tiu}\nl\infty{\partial_{Z_k}\oma})\hco{m-1}\oma\\
&\leq C(\uco2u\hco{m-1}\oma+\hco{m}u\lipco\oma)\hco{m-1}\oma
\end{align*}

We obtain from the previous relations the following differential inequality for the $H^{m-1}_{co}$ norm of $\oma$:
\begin{equation}\label{estoma}
\partial_t \hco{m-1}{\oma}^2
\leq C \| \oma\|_{H^{m-1}_{co}}^2(\nl\infty{\nabla u}+\uco2u)
+ C \| \oma\|_{H^{m-1}_{co}}\xm mu\lipco\oma.
\end{equation}

We recall now that the quantity $\nl2u^2+2\al\nl2{D (u)}^2$ is conserved. Let us introduce the following norm:
\begin{equation*}
\|u\|^2_{Y^m}\equiv\nl2u^2+2\al\nl2{D(u)}^2+\hco{m-1}{\oma}^2.
\end{equation*}
Then from Proposition \ref{elliptic} we have that $\xm mu\leq C\ym mu$. From \eqref{estoma} we infer that
\begin{equation*}
\partial_t \ym mu^2\leq C\ym mu^2(\nl\infty{\nabla u}+\uco2u+\lipco\oma).
\end{equation*}

From Lemma \ref{lemalip}  we deduce that
\begin{equation*}
 \nl\infty{\nabla u}+\uco2u\leq C(\nl\infty\om+\uco2u)
\end{equation*}
From the maximum principle applied to the operator $1-\al\Delta$ and using relation \eqref{boundombord} we deduce that
\begin{equation*}
  \nl\infty\om
  \leq \nl\infty{\oma}+\|\om\|_{L^\infty(\partial\Om)}
  \leq \nl\infty{\oma}+C\lipco u
\end{equation*}
so that
\begin{equation}\label{ominf}
 \nl\infty{\nabla u}+\uco2u\leq C(\nl\infty\oma+\uco2u)\leq C(\nl\infty\oma+\xm4u)\leq C(\nl\infty\oma+\ym mu)
\end{equation}
where we used the embedding $X^2\subset L^\infty$ proved in Lemma \ref{gagliardo}, item \ref{imbedding}). We conclude that
\begin{equation}\label{boundym}
\partial_t \ym mu\leq C\ym mu^2+C\ym mu\lipco\oma.
\end{equation}

It remains to estimate the $W^{1,\infty}_{co}$ norm of $\oma$. To do that, we use the equation for $\oma$ given in \eqref{eqvort}. We view it as a transport equation with source term $\oma\cdot\nabla u$. We have that
\begin{equation}\label{norminfoma}
\begin{aligned}
\nl\infty{\oma(t)}
&\leq  \nl\infty{\oma_0}+\int_0^t\nl\infty{\oma(s)}\nl\infty{\nabla u(s)}ds \\
&\leq  \nl\infty{\oma_0}+C\int_0^t\nl\infty{\oma(s)}(\nl\infty{\oma(s)}+\ym m{u(s)})ds \\
\end{aligned}
\end{equation}

Next, we apply a tangential derivative $\Z$ to \eqref{eqvort} and recall the decomposition $u=\sum\limits_{i=1}^7 \tiu_i Z_i$ to obtain
\begin{equation*}
  \dt\Z\oma+\Z(\sum_{i=1}^7  \tiu_i \partial_{Z_i}\oma)-\Z(\oma\cdot\nabla u)=0
\end{equation*}
so
\begin{equation*}
  \dt\Z\oma+u\cdot\nabla\Z\oma=-\sum_{i=1}^7  \Z\tiu_i \partial_{Z_i}\oma-\sum_{i=1}^7 \tiu_i [\Z,\partial_{Z_i}]\oma +\Z(\oma\cdot\nabla u).
\end{equation*}
We infer that
\begin{equation*}
\nl\infty{\Z\oma(t)} \leq  \nl\infty{\Z\oma_0}+\int_0^t\nl\infty{\sum_{i=1}^7  \Z\tiu_i \partial_{Z_i}\oma+\sum_{i=1}^7 \tiu_i [\Z,\partial_{Z_i}]\oma -\Z(\oma\cdot\nabla u)}.
\end{equation*}
Summing over all $Z$ and adding to  \eqref{norminfoma} we get the following bound for the $W^{1,\infty}_{co}$ norm of $\oma$:
\begin{multline*}
\lipco{\oma(t)}\leq \lipco{\oma_0} +C\int_0^t\nl\infty{\oma(s)}(\nl\infty{\oma(s)}+\ym m{u(s)})ds\\
+ C\int_0^t(\lipco{\tiu(s)}+\lipco{\nabla u(s)})\lipco{\oma(s)}\,ds.
\end{multline*}
Next, we estimate $\lipco{\tiu}\leq C\uco2u\leq C\xm4u\leq C\ym mu$. It remains to bound $\lipco{\nabla u}$. To do so, we use Lemma \ref{lemalip} and Lemma \ref{lemalipco} to write
\begin{align*}
\lipco{\nabla u}
&\leq C(\lipco\om+\uco2u)\\
&\leq  C(\lipco\oma+ \uco2u+\sqrt\al\uco3u+\al\uco4u) \\
&\leq C(\lipco\oma+ \ym mu+\sqrt\al\uco4u).
\end{align*}
The last term on the right-hand side can be estimated using Lemma \ref{lemmareg}, the relation \eqref{basic} and the embedding $X^2\subset L^\infty$:
\begin{equation}\label{m5}
\sqrt\al\uco4u\leq C\sqrt\al\xm6u\leq C\sqrt\al(\nl2u+\hco5\om)  \leq C(\nl2u+\hco4\oma)\leq C\ym mu.
\end{equation}
where we used that $m\geq5$. We conclude that
\begin{equation*}
\lipco{\oma(t)}\leq \lipco{\oma_0}
+ C\int_0^t(\lipco\oma+\ym mu)\lipco{\oma(s)}\,ds.
\end{equation*}
Combining the above relation with \eqref{boundym} integrated in time implies  that the quantity
\begin{equation*}
  F(t)=\ym m{u(t)}+\lipco{\oma(t)}
\end{equation*}
verifies the following relation
\begin{align*}
F(t)&\leq C\nl2{u_0}+C\sqrt\al\nl2{\nabla u_0}+C\hco{m-1}{\oma_0}+C\lipco{\oma_0}+C\int_0^t F^2(s)\,ds\\
&\leq C\nl2{u_0}+C\hco{m-1}{\oma_0}+C\lipco{\oma_0}+C\int_0^t F^2(s)\,ds
\end{align*}
where we also used Proposition \ref{elliptic}.
We finally deduce from Proposition \ref{elliptic} and from relation \eqref{ominf} that $L\leq CF\leq CL$. This completes the proof of Proposition \ref{apriori1}.
\end{proof}

We observe now that the previous \textit{a priori} estimates go through for solutions of the Euler equation, even with a small improvement.
\begin{proposition}\label{apriori2}
Let $u$ be a solution of the incompressible Euler equations \eqref{eulereq} with boundary conditions \eqref{tangent}. There exists a constant  $C=C(\Om)$ such that the following \textit{a priori} estimates hold true:
\begin{equation*}
M(t)\leq C(\nl2{u_0}+\|\om_0\|_{H^{3}_{co}\cap W^{1,\infty}_{co}})+C\int_0^tM^2(s)\,ds
\end{equation*}
where
\begin{equation*}
M=\|u\|_{X^4\cap W^{1,\infty}}+\|\om\|_{ W^{1,\infty}_{co}}.
\end{equation*}
\end{proposition}
\begin{proof}
We observe that even though now we don't assume $u$ to verify the Navier boundary conditions, the \textit{a priori} estimates proved in Proposition \ref{apriori1} remain valid when $\al=0$ too. Indeed, the only results from the previous sections that use the Navier boundary conditions are Proposition \ref{elliptic} and Lemma \ref{lemalipco}. But when $\al=0$ the conclusion of Lemma  \ref{lemalipco} is trivially true without requiring any boundary condition at all, and  the conclusion of Proposition \ref{elliptic} becomes the same as the conclusion of Lemma \ref{lemmareg}.

Moreover, if we go back to the proof of Proposition \ref{apriori1}, it is easy to see that the hypothesis $m\geq 5$ was used only in relation \eqref{m5}. In the rest of the proof the hypothesis $m\geq4$ is sufficient. But when $\al=0$, the relation \eqref{m5} is not required in the proof (and moreover it is trivially verified because the left-hand side vanishes). So in the case $\al=0$, the \textit{a priori} estimates proved in Proposition \ref{apriori1} are valid for $m=4$ and without need to assume the Navier boundary conditions. This completes the proof.
\end{proof}

\section{Approximation procedure and end of proofs}

In this section we construct an approximation procedure that will allow us to turn the \textit{a priori} estimates from the previous section into  a rigorous result of existence of solutions. We need to approximate the initial data by a sequence of smooth vector fields which belong to and are bounded in the same function spaces as $u_0$, that is, in conormal spaces. Density results for conormal spaces are known, see for example \cite{nishitani_regularity_1997,nishitani_regularity_2000}. However, these density results are false within the class of general divergence free vector fields. Indeed, it is proved in \cite{nishitani_regularity_1997,nishitani_regularity_2000} that $C^\infty_0$ is dense in $H^m_{co}$. A similar density result cannot be true for divergence free vector fields because a divergence free vector field has a normal trace at the boundary. If that normal trace is not vanishing, then no sequence of $C^\infty_0$ divergence free vector fields can converge to this vector field. In our case, a new approximation procedure must be invented and it is not at all obvious how to proceed.

Let $\PP$ be the Leray projector, \textit{i.e.} the $L^2$ orthogonal projection on the space of divergence free vector fields tangent to the boundary. The idea of our procedure of approximation of a divergence free vector field $\om$ in conormal spaces is given in the following lemma. It consists in observing that $\om -\PP\om$ belongs to the same Sobolev space as $\om$ but without the conormal subscript. So $\om -\PP\om$ can approximated using standard density results for the classical Sobolev spaces. As for $\PP u$, since it is tangent to the boundary the obstruction mentioned above disappears, and it is not hard to approximate it with smooth divergence free vector fields in conormal spaces.
\begin{lemma}\label{leray}
Let $m\geq2$ and $\om \in H^{m-1}_{co}(\Om)$ be a divergence free vector field. Then $\om -\PP \om \in H^{m-1}(\Om)$. Suppose in addition that $\om\in W^{1,\infty}_{co}$, that $m\geq4$ and that there exists some  $\psi$ such that $\om=\curl\psi$. Then there exist two vector fields $\psi_1$ and $\psi_2$ such that:
  \begin{gather}
    \om =\curl(\psi_1+\psi_2)\quad\text{and}\quad \psi_1+\psi_2=\psi-\nabla p\quad\text{for some }p,\label{f1}\\
\psi_1\in X^{m}\cap W^{2,\infty}_{co},\quad  \nabla\psi_1\in W^{1,\infty}_{co},\quad \dive\psi_1=0,\quad \psi_1\times n=0\text{ on }\partial\Om,\label{f2}\\
\psi_2\in H^m(\Om).\notag
  \end{gather}

\end{lemma}
\begin{proof}
We show first that $\om \cdot n\in X^{m-1}$. Because $H^{m-1}=H^{m-1}_{co}$ in the interior of $\Om$, it suffices to show it in $\Om_\delta$. But in that region we have that $\|n\|=1$, so
\begin{equation*}
\nabla=-n\times(n\times\nabla)+n\partial_n.
\end{equation*}
We infer that
\begin{equation*}
-[n\times(n\times\nabla)]\cdot \om   +n\cdot\partial_n\om =\dive \om =0
\end{equation*}
Clearly $n\cdot\partial_n\om =\partial_n(\om \cdot n)-\partial_nn\cdot \om $ so
\begin{equation*}
\partial_n(\om \cdot n)= \partial_nn\cdot \om  +[n\times(n\times\nabla)]\cdot \om .
\end{equation*}
The right-hand side belongs to $H^{m-2}_{co}$. We infer that  $\nabla(\om \cdot n)\in H^{m-2}_{co}$ so $\om \cdot n\in X^{m-1}$.

Now, let $\Z^{m-2}$ be a tangential derivative of order $\leq m-2$. Because $\om \cdot n\in X^{m-1}$ we have that $\Z^{m-2}(\om \cdot n)\in H^1(\Om)$ so $\Z^{m-2}(\om \cdot n)\bigl|_{\partial\Om}\in H^{\frac12}(\partial\Om)$. We conclude that $\om \cdot n\bigl|_{\partial\Om}\in H^{m-\frac32}(\partial\Om)$.

Next, from the properties of the Leray projector we know that there exists some $q\in H^1(\Om)$  such that
\begin{equation*}
  \om -\PP \om =\nabla q.
\end{equation*}
Recall that $\PP \om $ is divergence free and tangent to the boundary. Applying the divergence and taking the trace to the boundary of the above relation, we observe that $q$ verifies the following Neumann problem for the laplacian:
\begin{align*}
\Delta q&=0\quad \text{in }\Om\\
\partial_n q&=\om \cdot n\quad\text{on }\partial\Om.
\end{align*}
Because $\om \cdot n\bigl|_{\partial\Om}\in H^{m-\frac32}(\partial\Om)$, the classical regularity results for the Neumann problem of the laplacian imply that $q\in H^{m}(\Om)$. This completes the proof of the first part of the lemma.

\medskip

To prove the second part, let us define  $w= \om  -\PP \om  $.  From the first part of the lemma  we know that $w\in H^{m-1}$.  Since $m\geq4$, by Sobolev embedding we have that $H^{m-1}\subset W^{1,\infty}$ so we have in particular that $w\in H^{m-1}_{co}\cap W^{1,\infty}_{co}$. Since $ \om  $ also belongs to this space, we infer that $ \PP \om  \in H^{m-1}_{co}\cap W^{1,\infty}_{co}$.

Next, since $\PP \om  $ is divergence free and tangent to the boundary one can apply \cite[Theorem 2.1]{borchers_equations_1990} to find two vector fields $\overline\psi$ and $Y$ such that
\begin{gather*}
\PP\om=\curl \overline \psi+Y,\quad \overline\psi\bigl|_{\partial\Om}=0,\\
\dive Y=0,\quad \curl Y=0,\quad Y\cdot n  \bigl|_{\partial\Om}=0.
\end{gather*}
The vector field $Y$ is obviously smooth (as a consequence of \cite[Proposition 1.4]{foias_remarques_1978} for example).
Let $h$ be the solution of
\begin{align*}
\Delta h&=\dive \overline\psi\quad\text{in }\Om\\
h&=0 \quad\text{on }\partial\Om
\end{align*}
and let us define
$$\psi_1=\overline\psi-\nabla h.$$
Because $h$ vanishes on the boundary and $n\times\nabla$ are tangential derivatives, one has that $n\times\nabla h=0$ on the boundary. From the relations above one can readily check that $\psi_1$ has the following properties:
\begin{equation*}
\curl\psi_1=\PP \om-Y  ,\quad\dive\psi_1=0\quad\text{and}\quad \psi_1\times n=0\text{ on }\partial\Om.
\end{equation*}
Because $Y$ is smooth and $\PP\om\in H^{m-1}_{co}$ we infer that  $ \curl\psi_1  \in H^{m-1}_{co}$. As in Lemma \ref{lemmareg}, one can deduce  that $\psi_1\in X^m$. Indeed, the only difference between the setting of that lemma and the present setting is that in Lemma  \ref{lemmareg} the vector field is tangent to the boundary while here it is normal to the boundary. Nevertheless, the proof goes through by replacing the elliptic estimate given in \cite[Proposition 1.4]{foias_remarques_1978} with the elliptic estimate corresponding to normal vector fields given for instance in \cite[Corollary 2.15]{amrouche_vector_1998}. So we can conclude that $\psi_1\in X^m$. From the embedding $X^2\subset L^\infty$ we further obtain that $\psi_1 \in W^{2,\infty}_{co}$. Since $\PP\om\in W^{1,\infty}_{co}$ we have that $\curl \psi_1\in W^{1,\infty}_{co}$. Recalling that $\psi_1$ is divergence free, we infer from Lemma \ref{lemalip} that $\nabla\psi_1\in W^{1,\infty}_{co}$.
Relation \eqref{f2} is completely proved.

We define next
\begin{equation*}
\psi_2=\PP(\psi-\psi_1).
\end{equation*}
From the properties of the Leray projector we know that there is some $p$ such that
$$\psi-\psi_1-\psi_2=\psi-\psi_1-\PP(\psi-\psi_1)=\nabla p.$$
Taking the curl of the above equality shows that relation \eqref{f1} holds true. Finally, we observe that
\begin{equation*}
\curl\psi_2=\curl\psi-\curl\psi_1=\om -\PP \om +Y =w+Y\in H^{m-1}(\Om).
\end{equation*}
Recalling that $\psi_2$ is also divergence free and tangent to the boundary, we can  apply \cite[Proposition 1.4]{foias_remarques_1978} to deduce that $\psi_2\in H^m$. This completes the proof.
\end{proof}

In the next proposition we use the previous lemma to construct a sequence of smooth approximations of the initial data.

\begin{proposition}\label{approxnavier}
Let $u$ be a divergence free vector field verifying the Navier boundary conditions \eqref{navier} and such that $u\in H^2$ and $\oma\in H^{m-1}_{co}\cap W^{1,\infty}_{co}$ where $m\geq4$. There exists a sequence of smooth divergence free vector fields $u_n$  verifying the Navier boundary conditions such that $u_n\to u$ in $H^2$ and such that
\begin{equation}\label{estaprox}
\nl2{u_n}+\hco{m-1}{\oma_n}+\lipco{\oma_n}\leq C(\nl2{u}+\hco{m-1}{\oma}+\lipco{\oma})
\end{equation}
for some constant $C=C(m,\Om)$.
\end{proposition}
\begin{proof}
Let $v=u-\al\Delta u$ so that $\oma=\curl v$. Because $\oma$ is divergence free, we can apply the previous lemma to $\oma$ to deduce the existence of some vector fields $\psi_1$ and $\psi_2$ such that
  \begin{gather*}
    \oma=\curl(\psi_1+\psi_2)\quad\text{and}\quad \psi_1+\psi_2=v-\nabla p\quad\text{for some }p,\\
\psi_1\in X^{m}\cap W^{2,\infty}_{co},\quad  \nabla\psi_1\in  W^{1,\infty}_{co},\quad \dive\psi_1=0,\quad \psi_1\times n=0\text{ on }\partial\Om,\\
\psi_2\in H^m(\Om).
  \end{gather*}

Let $\varphi:\R_+\to[0,1]$ be a smooth function such that $\varphi(s)=1$ pour $s>1$ and $\varphi(s)=0$ for $s<1/2$. We define $\varphi_\ep(x)=\varphi(d/\ep)$ and $\psi_1^\ep=\varphi_\ep\psi_1$. Clearly $\psi_1^\ep\to\psi_1$ in $L^2$ as $\ep\to0$. Moreover, we claim that $\curl \psi_1^\ep$ is bounded in $H^{m-1}_{co}\cap W^{1,\infty}_{co}$ uniformly in $\ep$. To prove this, we start by writing
\begin{equation*}
\curl  \psi_1^\ep=\varphi_\ep\curl\psi_1-\psi_1\times\nabla\varphi_\ep
=\varphi_\ep\curl\psi_1-\frac1\ep\psi_1\times\nabla d\ \varphi'\Bigl(\frac d\ep\Bigr).
\end{equation*}

We remark now that for every $k\in\N$ the functions $\varphi_\ep$ are bounded in  $W^{k,\infty}_{co}$ uniformly in $\ep$. Indeed, if $\Z$ is a tangential derivative, we have that
\begin{equation*}
  \Z\varphi_\ep=\frac{\Z d}\ep\varphi'\Bigl(\frac d\ep\Bigr).
\end{equation*}
Since $d$ vanishes on the boundary and $\Z$ is a tangential derivative we have that $\Z d$ vanishes on the boundary. Because the support of $\varphi'(d/\ep)$ is included in $\Om_\ep$ for $\ep$ sufficiently small, the mean value theorem implies that $|\Z d|\leq C\ep\|d\|_{W^{2,\infty}(\Om_\delta)}$ on the support of  $\varphi'(d/\ep)$ (we assumed that $\ep$ is sufficiently small). So $\Z\varphi_\ep$ is uniformly bounded in $\ep$ and a similar argument works for the higher order tangential derivatives of $\varphi_\ep$.

Since $\varphi_\ep$ is bounded in  $W^{k,\infty}_{co}$ uniformly in $\ep$ and $\curl\psi_1\in H^{m-1}_{co}\cap W^{1,\infty}_{co}$, the Leibniz formula immediately implies that $\varphi_\ep\curl\psi_1$ is bounded in $ H^{m-1}_{co}\cap W^{1,\infty}_{co}$ uniformly in $\ep$.

We remark next that since $d$ vanishes on the boundary, its gradient is normal to the boundary. But $\psi_1$ is also normal to the boundary, so $\psi_1\times\nabla d$ vanishes on the boundary. We can therefore apply Lemma \ref{clar} to deduce that
\begin{align*}
C\bigl\|\frac{\psi_1\times\nabla d}d\bigr\|_{H^{m-1}_{co}\cap W^{1,\infty}_{co}}
&\leq C(\|\psi_1\times\nabla d\|_{H^{m-1}_{co}\cap W^{1,\infty}_{co}}+\|\partial_n(\psi_1\times\nabla d)\|_{H^{m-1}_{co}\cap W^{1,\infty}_{co}})\\
&\leq C(\xm m{\psi_1}+\lipco{\psi_1}+\lipco{\nabla\psi_1}).
\end{align*}
As above, one can easily check that $\frac d\ep \varphi'(\frac d\ep)$ is bounded independently of $\ep$ in any $W^{k,\infty}_{co}$. We conclude by the Leibniz formula that $\frac1\ep\psi_1\times\nabla d\ \varphi'(\frac d\ep)=\frac{\psi_1\times\nabla d}d \ \frac d\ep  \varphi'(\frac d\ep)$ is bounded independently of $\ep$ in  $ H^{m-1}_{co}\cap W^{1,\infty}_{co}$.

We infer from the previous relations that $\curl\psi_1^\ep$ is bounded independently of $\ep$ in  $ H^{m-1}_{co}\cap W^{1,\infty}_{co}$. Since $\psi_1^\ep$ is compactly supported in $\Omega$, it can be smoothed out by convolution with an approximation of the identity. Letting $\ep\to0$ afterwards, one can construct a sequence $\psi_1^n$ of smooth vector fields such that $\psi_1^n\to\psi_1$ in $L^2$ and such that $\curl\psi_1^n$ is bounded in  $ H^{m-1}_{co}\cap W^{1,\infty}_{co}$.

Next, by density of smooth functions in $H^m$, there exists a sequence of smooth vector fields $\psi_2^n$ such that $\psi_2^n\to \psi_2$ in $H^m$. Since $m\geq4$ we have the Sobolev embedding $H^m\subset W^{2,\infty}$ so $\curl\psi_2^n$ is bounded in  $ H^{m-1}\cap W^{1,\infty}$.
Let $v_n=\psi_1^n+\psi_2^n$. Then $v_n\to \psi_1+\psi_2$ in $L^2$ and $\curl v_n$ is bounded in  $ H^{m-1}_{co}\cap W^{1,\infty}_{co}$. Let $u_n$ be the solution of the following Stokes problem:
\begin{equation*}
  u_n-\al\Delta u_n=v_n+\nabla p_n, \quad
\dive u_n=0, \quad u_n\text{ verifies the Navier boundary conditions \eqref{navier}}.
\end{equation*}
Since $\psi_1+\psi_2=v-\nabla p$, we observe that $u$ verifies the following Stokes problem:
\begin{equation*}
  u-\al\Delta u=\psi_1+\psi_2+\nabla p, \quad
\dive u=0, \quad u\text{ verifies the Navier boundary conditions \eqref{navier}}.
\end{equation*}

But regularity results for the above Stokes problem are known. We can deduce for instance from \cite[Theorem 3]{busuioc_second_2003} that $\|u_n-u\|_{H^2}\leq \nl2{v_n-\psi_1-\psi_2}\to0$. Since $v_n$ is smooth, the same theorem also implies that $u_n$ is smooth. Moreover, $\om_n^\al=\curl (u_n-\al\Delta u_n)=\curl v_n$ is bounded in  $ H^{m-1}_{co}\cap W^{1,\infty}_{co}$. Finally, one can also easily keep track of the estimates in the above arguments and deduce that relation \eqref{estaprox} holds true for some constant $C$. The sequence $u_n$ has all required properties and this completes the proof.
\end{proof}

This proposition allows us to finish the proof of Theorem \ref{uniformtime}.
\begin{proof}[Proof of Theorem \ref{uniformtime}]
From the previous proposition, we deduce the existence of a sequence of smooth  velocity fields $u_0^n$ verifying the Navier boundary conditions such that $u^n_0\to u_0$ in $H^2$ and such that
\begin{equation*}
\nl2{u^n_0}+\hco{m-1}{\om^{\al,n}_0}+\lipco{\om^{\al,n}_0}\leq C(\nl2{u_0}+\hco{m-1}{\oma_0}+\lipco{\oma_0})
\end{equation*}

Using the result of \cite{busuioc_second_2003}, one can construct a local solution $u^n$ with initial velocity  $u_0^n$. This solution is smooth. Indeed, even though the result of \cite{busuioc_second_2003} is stated only in $H^3$ it easily goes through to any $H^m$ with $m\geq3$. Moreover, the blow-up of the solution cannot occur while the Lipschitz norm of the solution is bounded. On these smooth solutions, the \textit{a priori} estimates proved in Proposition \ref{apriori1}  are valid. Therefore, we obtain a bound on the quantity $L(t)$ on a time interval $[0,T_n]$ of size
\begin{equation*}
T_n=\frac C{\nl2{u^n_0}+\hco{m-1}{\om^{\al,n}_0}+\lipco{\om^{\al,n}_0}}\geq  \frac C{\nl2{u_0}+\hco{m-1}{\om^{\al}_0}+\lipco{\om^{\al}_0}}\equiv T.
\end{equation*}

In particular, we control the Lipschitz norm of the solution on $[0,T]$ where $T$ does not depend on $n$. Because the Lipschitz norm of $u_n$ is bounded uniformly in $n$ on the time interval $[0,T]$, we infer that the solution of \eqref{maineq} and \eqref{navier} exists at least up to the time $T$. Finally, given that the solutions are  bounded in $H^3$ with respect to $n$, passing to the limit as $n\to\infty$ on $[0,T]$ is quite simple and standard. This completes the proof of Theorem  \ref{uniformtime}.
\end{proof}

As mentioned in the introduction, Theorem \ref{convergence} is a direct consequence of Theorem \ref{uniformtime} and of \cite[Theorem 5]{busuioc_incompressible_2012}.

Finally, to complete the proof of Theorem \ref{euler} one can turn in a similar manner the \textit{a priori} estimates of Proposition \ref{apriori2} into a rigorous result of existence of solutions provided that we can construct a suitable sequence of smooth velocity fields approximating the initial velocity field. This is performed in the next proposition.

\begin{proposition}\label{approxeuler}
Let $u\in X^m$, $m\geq4$, be a divergence free vector field tangent to the boundary such that $\om=\curl u \in W^{1,\infty}_{co}$. There exists a sequence of smooth divergence free vector fields $u_n$  tangent to the boundary such that $u_n\to u$ in $L^2$ and such that
\begin{equation}\label{boundeuler}
\xm m{u_n}+\lipco{\om_n}\leq C(\xm mu +\lipco{\om})
\end{equation}
for some constant $C=C(m,\Om)$.
\end{proposition}
\begin{proof}
We apply Lemma \ref{leray} to $\om$ to find two vector fields $\psi_1$ and $\psi_2$ such that:
  \begin{gather*}
    \om =\curl(\psi_1+\psi_2)\quad\text{and}\quad \psi_1+\psi_2=u-\nabla p\quad\text{for some }p,\\
\psi_1\in X^{m}\cap W^{2,\infty}_{co},\quad  \nabla\psi_1\in W^{1,\infty}_{co},\quad \dive\psi_1=0,\quad \psi_1\times n=0\text{ on }\partial\Om,\\
\psi_2\in H^m(\Om).
  \end{gather*}

As in the proof of Proposition \ref{approxnavier}, we can construct  two sequences of smooth vector fields $\psi_1^n$ and $\psi_2^n$  such that
\begin{itemize}
\item $\psi_1^n\to\psi_1$ in $L^2$;
\item $\curl\psi_1^n$ is bounded in  $ H^{m-1}_{co}\cap W^{1,\infty}_{co}$;
\item  $\psi_2^n\to \psi_2$ in $H^m$;
\item  $\curl\psi_2^n$ is bounded in  $ H^{m-1}\cap W^{1,\infty}$.
\end{itemize}

We define
\begin{equation*}
  u_n=\PP(\psi_1^n+\psi_2^n).
\end{equation*}
Because $\psi_1^n$ and $\psi_2^n$ are bounded in $L^2$, so is $u_n$. Moreover, since $u_n$ and $\psi_1^n+\psi_2^n$ differ by a gradient we have that $\om_n=\curl u_n=\curl (\psi_1^n+\psi_2^n)$ is bounded in $ H^{m-1}_{co}\cap W^{1,\infty}_{co}$. From Lemma \ref{lemmareg} we infer that $u_n$ is bounded in $X^m$. Next, since $\PP$ is a  bounded operator on $L^2$ we have that
\begin{equation*}
  \lim_{n\to\infty}u_n=\lim_{n\to\infty}\PP(\psi_1^n+\psi_2^n)=\PP\lim_{n\to\infty}(\psi_1^n+\psi_2^n)
=\PP(\psi_1+\psi_2)=\PP(u-\nabla p)=u\quad\text{in }L^2.
\end{equation*}
Keeping track of the estimates one can deduce relation \eqref{boundeuler} for some constant $C$. This completes the proof.
\end{proof}

We can now complete the proof of the last theorem in this paper.
\begin{proof}[Proof of Theorem \ref{euler}]
According to Proposition \ref{approxeuler}, there exists a sequence $u^n_0$ of smooth divergence free vector fields tangent to the boundary such that $u^n_0\to u_0$ in $L^2$ and
\begin{equation*}
\xm4{u^n_0}+\lipco{\om^n_0}\leq C(\xm4{u_0}+\lipco{\om_0}).
\end{equation*}
One can construct a smooth local in time solution $u^n$ of the Euler equation \eqref{eulereq}, \eqref{tangent} with initial data $u_0$. By the Beale-Kato-Majda criterion, the solution does not blow-up as long as the Lipschitz norm of $u_n$ does not blow-up. The \textit{a priori} estimates of Proposition \ref{apriori2} hold true. By the Gronwall lemma, the quantity $M(t)$ stays bounded on a time interval $T_n$ such that
\begin{equation*}
T_n=\frac C{\nl2{u^n_0}+\|\om^n_0\|_{H^3_{co}\cap W^{1,\infty}_{co}}}
\geq \frac C{\xm4{u^n_0}+\lipco{\om^n_0}}
\geq \frac C{\xm4{u_0}+\lipco{\om_0}}
\equiv T.
\end{equation*}
So the solution $u^n$ exists up to time $T$ and its Lipschitz norm is bounded on $[0,T]$. Then one can easily pass to the limit and show that $u^n$ converges to a solution of the Euler equation with the required properties.
\end{proof}

\section{A final remark}

We begin this section with the observation that it is possible to extend our result on the uniform time of existence to the second grade fluid equations, given by
\begin{equation}\label{secgrade}
  \partial_t (u-\al\Delta u)-\nu\Delta u+u\cdot\nabla(u-\al\Delta u)+\sum_j(u-\al\Delta u)_j\nabla u_j=-\nabla p, \qquad \dive u=0,
\end{equation}
as long as  $\nu / \al$ is bounded. Indeed, the vorticity equation can be written under the form
\begin{equation*}
  \dt\oma+\frac\nu\al\oma-\frac\nu\al\om+u\cdot\nabla\oma-\oma\cdot\nabla u=0.
\end{equation*}
If $\nu/\al$ is bounded, then the two additional terms are not worse than the others so that estimates similar to the ones developed above hold true, giving the same results.
Putting this together with \cite[Theorem 5]{busuioc_incompressible_2012} we obtain that, under this restriction on $\nu$, $\al$, the limit of solutions of the second grade fluid equations is a solution of the Euler equations.

Note that the second grade fluid equations are an interpolant between the Navier-Stokes equations ($\al = 0$) and the $\al$-Euler equations ($\nu=0$). The work by Masmoudi and Rousset refer to the  extremal $\al=0$, while the results contained in our paper, together with the extension discussed above, correspond to the cases $\nu=\mathcal{O}(\al)$. This raises the possibility that, combining our arguments with those of \cite{masmoudi_uniform_2012-1}, a general result for the second grade fluid equations could be obtained.

One additional problem left open is to extend this work to Navier boundary conditions with nonzero friction coefficient, such as were treated in \cite{masmoudi_uniform_2012-1}.

\bigskip

\subsection*{Acknowledgments} The authors would like to thank an anonymous referee for \cite{busuioc_incompressible_2012} for the suggestion to use conormal Sobolev spaces. A.V.B. and D.I. are grateful for the hospitality of the Universidade Federal do  Rio de Janeiro and of IMPA, while M.C.L.F. and H.J.N.L. thank the hospitality of the Université Claude Bernard Lyon 1. This work was funded in part by the Réseau Franco-Brésilien en Mathématiques.  D.I. has been partially funded by the ANR project Dyficolti ANR-13-BS01-0003-01. H.J.N.L.'s research has been funded in part by CNPq Grant \# 307918/2014-9 and by FAPERJ Grant \# E-26/103.197/2012. The work of M.C.L.F. has been partially funded by CNPq Grant \# 306886/2014-6.

\bibliographystyle{myabbrven}
\bibliography{zotero}
\end{document}